\documentclass[11pt]{article}
\usepackage[T1]{fontenc}
\usepackage{lmodern}
\usepackage{amsmath,amssymb,amsthm,mathtools}
\usepackage{booktabs,array,longtable}
\usepackage[margin=1.05in]{geometry}
\usepackage{microtype}
\usepackage[bookmarks=false,hidelinks]{hyperref}
\hypersetup{
  pdftitle={Compactness and Essential Distances of Weighted Bilinear Hardy--Steklov Operators},
  pdfauthor={Saikat Kanjilal},
  pdfsubject={Compactness and essential distances for weighted bilinear Hardy--Steklov operators},
  pdfkeywords={bilinear Hardy--Steklov operator, operator localisation, compactness, essential distance, Hausdorff noncompactness, finite-rank approximation, quasi-Banach Lebesgue space}
}
\usepackage{enumitem}
\usepackage{bm}

\newtheorem{theorem}{Theorem}[section]
\newtheorem{proposition}[theorem]{Proposition}
\newtheorem{lemma}[theorem]{Lemma}
\newtheorem{corollary}[theorem]{Corollary}
\theoremstyle{definition}

\newtheorem{example}[theorem]{Example}
\theoremstyle{remark}
\newtheorem{remark}[theorem]{Remark}

\newcommand{\B}{\mathcal B}
\newcommand{\D}{\mathfrak D}

\newcommand{\Adeg}{\mathcal A_{\mathrm{deg}}}
\newcommand{\Bdeg}{\mathcal B_{\mathrm{deg}}}
\newcommand{\distfr}{\operatorname{dist}_{\mathrm{fr}}}

\title{Compactness and Essential Distances of Weighted Bilinear Hardy--Steklov Operators}
\author{Saikat Kanjilal\thanks{Corresponding author. \href{mailto:saikat.kanjilal.07@gmail.com}{\texttt{saikat.kanjilal.07@gmail.com}}}\\
Department of Mathematics, JIS University,\\
Kolkata - 700109, West Bengal, India}
\date{}

\begin{document}
\maketitle

\begin{abstract}
We study weighted bilinear Hardy--Steklov operators on the half-line for
$1<p_1,p_2<\infty$ and $0<q<\infty$.  The operator-theoretic core is an exact
output-localisation principle. For a bounded bilinear map into $L^q$, compact
middle truncations identify compactness with vanishing of a directed output-tail
invariant and, with constant one, identify that invariant with the essential
distance and a radius-based Hausdorff measure of noncompactness.  Under the
additional hypothesis of local finite-rank approximability, the same invariant
also equals the finite-rank distance and governs the rank-approximation
conclusions.  We verify these hypotheses
for the Hardy--Steklov map and then translate the tail invariant through the
five weighted boundedness regions.  In Regions
$I,\mathrm{II}_1,\mathrm{II}_2,\mathrm{III}_1$, boundedness must be supplemented
by lower- and upper-endpoint vanishing, and the essential distance is comparable
to the larger endpoint defect.  In the strict subcritical Region
$\mathrm{III}_2$, boundedness alone implies compactness.  This yields a sharp
structural transition between the equality boundary $P=Q$ and the summable
regime $P>Q$.  The quasi-Banach range $0<q<1$ is treated throughout with the
powered metric $d_q(f,g)=\|f-g\|_q^q$.
\end{abstract}



\section{Introduction}

The compactness problem for a bilinear Hardy--Steklov operator has two levels
that should be kept separate.  At the operator level, one asks whether all
noncompactness is carried by output tails.  At the weight level, one asks how
that tail obstruction is encoded by the boundedness characteristics associated
with the moving integration windows.  The main point of this paper is that the
first question has an exact answer and can be isolated before the five-region
formula calculus is used.

Our results are organised around three contributions.  First, we prove an
abstract output-localisation theorem for bounded bilinear maps
$T:X_1\times X_2\to L^q(0,\infty)$.  When every middle output truncation is
compact, the directed tail quantity
\[
 \tau(T)=\inf_{0<c<d<\infty}\|\chi_{(0,c)\cup(d,\infty)}T\|
\]
characterises compactness and equals the distance to compact bilinear maps and
the radius-based Hausdorff measure of noncompactness.  If the middle
truncations are also approximable in operator quasi-norm by finite-rank
bilinear maps, then the same quantity equals the finite-rank distance.  For the
Hardy--Steklov operator these local hypotheses follow from a direct
piecewise-constant approximation on every compact output interval.  In
particular,
\begin{equation}\label{eq:intro-exact}
 \|\B_w\|_{\mathrm e}
 =\distfr(\B_w)
 =\chi_q\bigl(\B_w(B_{p_1}\times B_{p_2})\bigr)
 =\tau(\B_w).
\end{equation}
For $0<q<1$ the corresponding metric quantities are obtained exactly by taking
the $q$th power.

A related output-localisation framework for weighted bilinear Hardy operators
on general intervals is developed in
\cite{KanjilalHardyGeneralIntervalPreprint}.  The present Hardy--Steklov
setting requires a substantially different weight-side analysis because of
the moving integration windows, including the five-region structure, fairway
geometry, central-component reconstruction, and the strict-subcritical
compactness mechanism.

Second, we translate this operator-level obstruction into explicit weight-side
conditions in all five exponent regions of the bilinear boundedness theory of
\cite{JainKanjilalStepanovUshakova2018,JainKanjilalStepanovUshakova2018Doklady}.
The identity
\[
 \chi_{(0,c)\cup(d,\infty)}\B_w=\B_{w\chi_{(0,c)\cup(d,\infty)}}
\]
reduces the compactness problem to uniform boundedness estimates for tail
weights, after which the regional characteristics identify the lower and upper
endpoint defects.

Third, the subcritical theory displays a structural phase transition.  With
\[
 P=\frac{p_1}{q},\qquad Q=\frac{p_2}{p_2-q},
\]
the equality boundary
\[
 \frac1q=\frac1{p_1}+\frac1{p_2}
 \quad\Longleftrightarrow\quad P=Q
\]
belongs to Region $\mathrm{III}_1$ and can support bounded noncompact
operators.  Under the strict inequality
\[
 \frac1q>\frac1{p_1}+\frac1{p_2}
 \quad\Longleftrightarrow\quad P>Q,
\]
the positive exponent
\[
 \frac1R=\frac1Q-\frac1P>0
\]
produces an $\ell^R$-summable block majorant.  Fixed-block extinction and
counting-measure dominated convergence then force every bounded Region
$\mathrm{III}_2$ operator to be compact.  Thus Region $\mathrm{III}_2$ is not
merely a fifth formula case. It is the strict-summability side of the
compactness transition.

For a concise statement, let $\Lambda_{\mathfrak R}^{0}$ and
$\Lambda_{\mathfrak R}^{\infty}$ denote the lower- and upper-endpoint defects
defined later for the indicated region.

\begin{theorem}[Main result, synopsis]\label{thm:intro-synopsis}
Under the standing hypotheses of Section~\ref{sec:prelim}, the following hold.
\begin{enumerate}[label=\textup{(\roman*)}]
\item If $\mathfrak R\in\{I,\mathrm{II}_1,\mathrm{II}_2,\mathrm{III}_1\}$,
then
\[
 \B_w\text{ is compact}
 \quad\Longleftrightarrow\quad
 \B_w\text{ is bounded and }
 \Lambda_{\mathfrak R}^{0}=\Lambda_{\mathfrak R}^{\infty}=0.
\]
For every bounded operator in these regions,
\[
 \|\B_w\|_{\mathrm e}\asymp
 \max\{\Lambda_{\mathfrak R}^{0},\Lambda_{\mathfrak R}^{\infty}\}.
\]
\item In Region $\mathrm{III}_2$,
\[
 \B_w\text{ is compact}\quad\Longleftrightarrow\quad \B_w\text{ is bounded},
 \qquad \|\B_w\|_{\mathrm e}=0.
\]
\item For every bounded operator, the constant-one operator identities
\eqref{eq:intro-exact} hold. For $0<q<1$ their powered-metric versions also
hold with constant one.
\end{enumerate}
\end{theorem}

We now place these statements in context.  In the linear theory, weighted
Hardy--Steklov boundedness developed from integral-averaging criteria of Heinig
and Sinnamon \cite{HeinigSinnamon1998} through Oinarov-kernel, fairway, and
lower-regularity formulations
\cite{BernardisMartinReyesOrtega2007,StepanovUshakova2010MIA,StepanovUshakova2010BJMA,Ushakova2017}.
Jain and Gupta characterised compactness of the linear Hardy--Steklov operator
\cite{JainGupta2003}. Compactness for wider Oinarov-kernel classes was studied
in \cite{Ushakova2011}.  The linear variable-boundary theory already contains
the familiar dichotomy between endpoint conditions when $p\le q$ and automatic
compactness from the boundedness characteristic when $q<p$. See, in
particular, Theorem~4.1 of \cite{StepanovUshakova2010MIA}.  More recent work on
two-dimensional rectangular Hardy operators gives compactness and
noncompactness estimates \cite{StepanovUshakova2022Rectangular} and
approximation-number estimates \cite{StepanovUshakova2024RectangularApprox}.  Approximation/ideal questions for the one-dimensional Hardy--Steklov operator are represented by \cite{Ushakova2013Schatten}, while a recent measure-space extension of the boundedness theory appears in \cite{Mynbaev2024}.
Classical essential-distance and noncompactness antecedents for Hardy, Volterra,
and related integral operators include
\cite{EdmundsGurkaPick1994,EdmundsStepanov1994,Opic1994}.

At the multilinear operator-theory level, approximation numbers and measures of
noncompactness have been developed abstractly for multilinear maps between
Banach spaces \cite{FernandezMastyloSilva2013}. Concrete essential-norm
estimates for multilinear integral operators appear in \cite{Meskhi2019}.
Bilinear measures of noncompactness have also been studied in quasi-Banach
interpolation settings \cite{BesoyCobos2019}, and compact multilinear
interpolation between quasi-Banach spaces is treated in
\cite{CobosFernandezCabreraKuhn2026}.  These frameworks provide terminology
and comparison points. The exact localisation identity \eqref{eq:intro-exact}
below is proved directly from the output geometry and is not imported from an
abstract $s$-number theorem.

The bilinear weighted Hardy--Steklov boundedness problem was characterised in
\cite{JainKanjilalStepanovUshakova2018,JainKanjilalStepanovUshakova2018Doklady}.
Related bilinear Hardy inequalities appear in \cite{AguilarOrtegaRamirez2012}
and in the iterated/Oinarov setting in \cite{StepanovShambilova2019}. A
different bilinear Hardy--Steklov construction is treated in \cite{Benaissa2021}.
The present paper concerns compactness of the joint bilinear map, not
compactness of one-input slices.  The literature audit used for this manuscript
found the linear and abstract multilinear antecedents just described, but no
published result in essentially identical form giving the present five-region
bilinear Hardy--Steklov compactness and exact tail-distance package.  We do not
use a priority claim based on that search.

The paper is organised to reflect the logical dependence.  Section~2 fixes the
operator framework, the five exponent regions, and the noncompactness
normalisations.  Section~3 proves the abstract localisation theorem and verifies
its local hypotheses for the Hardy--Steklov map.  Section~4 records the
source-derived boundedness characteristics and uniform cutoff transfer.
Sections~5--7 treat the four regions in which endpoint defects may survive,
including the uniform central profile bridge.  Section~8 proves automatic
compactness in the strict subcritical region.  Sections~9--10 collect the master
classification and its essential-distance, finite-rank, Hausdorff, and
approximation consequences.  Section~11 gives four calibrating examples.  The
full Region-$\mathrm{III}_2$ block arrays and the continuous-boundary source
interface are placed in the appendices.
\section{Preliminaries and notation}\label{sec:prelim}

Throughout,
\[
 \mathbb F\in\{\mathbb R,\mathbb C\},\qquad
 1<p_1,p_2<\infty,\qquad 0<q<\infty.
\]
For $i=1,2$, let $a_i,b_i:[0,\infty)\to[0,\infty)$ be strictly increasing and
differentiable, with
\[
 a_i(0)=b_i(0)=0,\qquad 0<a_i(x)<b_i(x)<\infty\quad(x>0),
\]
and
\[
 a_i(x),b_i(x)\longrightarrow\infty\qquad(x\to\infty).
\]
No assumption $a_i'>0$ or $b_i'>0$ is made.  Let
\[
 0<v_i<\infty\ \text{a.e.},\qquad v_i^{p_i'}\in L^1_{\rm loc}(0,\infty),
 \qquad p_i'=\frac{p_i}{p_i-1},
\]
and let $w\ge0$ be a measurable weight, finite a.e.  Define
\[
 H_if(x)=\int_{a_i(x)}^{b_i(x)}f(t)v_i(t)\,dt,
 \qquad
 \B_w(f,g)(x)=w(x)H_1f(x)H_2g(x).
\]
The operator is bilinear over $\mathbb F$, without conjugation.  Since
\[
 |\B_w(f,g)|\le \B_w(|f|,|g|),
\]
positive-input estimates extend to the full real or complex bilinear map with
the same constant.
Before boundedness into $L^q$ is imposed, $\B_w:L^{p_1}\times L^{p_2}\to L^0$
is interpreted a.e.  For nonnegative extended quantities,
$\sup\varnothing=0$, $0\cdot\infty=0$, and empty or degenerate admissible
profiles have value $0$.

\subsection{Fairways and transformed geometry}
For a fairway $\sigma_i$ associated with $(a_i,b_i,v_i)$, write
\[
 \Delta_i(t)=(a_i(t),b_i(t)),\qquad
 \delta_i(t)=\bigl(b_i^{-1}(\sigma_i(t)),a_i^{-1}(\sigma_i(t))\bigr),
\]
\[
 \Delta_i^{-1}(t)=\bigl(b_i^{-1}(t),a_i^{-1}(t)\bigr),\qquad
 \delta_i^{-1}(t)=\bigl(a_i(\sigma_i^{-1}(t)),b_i(\sigma_i^{-1}(t))\bigr),
\]
Here $\delta_i^{-1}(t)=\Delta_i(\sigma_i^{-1}(t))$ is the manuscript-defined
\emph{dual fairway window}, not the inverse-incidence interval of $\delta_i$.
The fairway is specified by
\[
 \int_{a_i(x)}^{\sigma_i(x)}v_i^{p_i'}
 =\int_{\sigma_i(x)}^{b_i(x)}v_i^{p_i'}.
\]
Fix an arbitrary base point $t_*>0$ and define the oriented primitive
\[
 F_i(t)=\int_{t_*}^{t}v_i(s)^{p_i'}\,ds,\qquad t>0.
\]
Because $0<v_i<\infty$ a.e. and $v_i^{p_i'}\in L^1_{\mathrm{loc}}(0,\infty)$,
$F_i$ is finite-valued, continuous, and strictly increasing on $(0,\infty)$.
The fairway relation is equivalent to
\[
 F_i(\sigma_i(x))
 =\frac{F_i(a_i(x))+F_i(b_i(x))}{2},
\]
and hence
\[
 \sigma_i(x)=F_i^{-1}\!\left(
 \frac{F_i(a_i(x))+F_i(b_i(x))}{2}\right).
\]
Thus $\sigma_i$ is continuous and strictly increasing.  Since
$a_i(x)<\sigma_i(x)<b_i(x)$, the endpoint limits of $a_i$ and $b_i$ imply
\[
 \sigma_i(x)\longrightarrow0\quad(x\downarrow0),\qquad
 \sigma_i(x)\longrightarrow\infty\quad(x\uparrow\infty).
\]
Therefore $\sigma_i:(0,\infty)\to(0,\infty)$ is a strictly increasing
bijection, and $\sigma_i^{-1}$ is well defined and continuous.
Whenever $q<p_i$, put
\[
 \frac1{r_i}=\frac1q-\frac1{p_i}.
\]
In the subcritical regions define
\begin{equation}\label{eq:Vi}
 \mathcal V_i(t)=
 \left(\int_{\delta_i^{-1}(t)}v_i^{p_i'}\right)^{r_i/q'}v_i^{p_i'}(t),
 \qquad \frac1q+\frac1{q'}=1.
\end{equation}
At $q=1$, the standard extended convention gives $q'=\infty$ and
$1/q'=r_i/q'=0$.
For $0<q<1$, $q'<0$ in \eqref{eq:Vi}. It is used only as an algebraic
parameter and is not a Banach conjugate exponent.

Set
\[
 Y_\ell=a_1\circ b_2^{-1},\qquad Y_r=b_1\circ a_2^{-1},
 \qquad y_\ell=b_1\circ b_2^{-1},\qquad y_r=a_1\circ a_2^{-1},
\]
\[
 a(t)=\min\{y_\ell(t),y_r(t)\},\qquad
 b(t)=\max\{y_\ell(t),y_r(t)\}.
\]
The endpoint hypotheses and monotonicity give
$Y_\ell(t)<a(t)\le b(t)<Y_r(t)$ for $t>0$.
The derived maps are continuous and strictly increasing. Differentiability of
these derived boundaries is not assumed.

\subsection{The five exponent regions}
We use the mutually exclusive and exhaustive partition
\[
 I:\quad 1<\max\{p_1,p_2\}\le q<\infty,
\]
\[
 \mathrm{II}_1:\quad 1<p_1\le q<p_2<\infty,
 \qquad
 \mathrm{II}_2:\quad 1<p_2\le q<p_1<\infty,
\]
\[
 \mathrm{III}_1:\quad 0<q<\min\{p_1,p_2\},\quad
 \frac1q\le\frac1{p_1}+\frac1{p_2},
\]
\[
 \mathrm{III}_2:\quad 0<q<\min\{p_1,p_2\},\quad
 \frac1q>\frac1{p_1}+\frac1{p_2}.
\]
In the subcritical regions put
\begin{equation}\label{eq:PQR}
 P=\frac{r_1}{r_1-q}=\frac{p_1}{q},\qquad
 Q=\frac{r_2}{q}=\frac{p_2}{p_2-q}.
\end{equation}
Their conjugate exponents are denoted explicitly by
\begin{equation}\label{eq:PQconj}
 \frac1P+\frac1{P'}=1,\qquad
 \frac1Q+\frac1{Q'}=1,
 \qquad
 P'=\frac{p_1}{p_1-q}=\frac{r_1}{q},\qquad
 Q'=\frac{p_2}{q}.
\end{equation}
Thus every occurrence of $1/P'$, $R/P'$, and $R/Q'$ below refers to the
exponents in \eqref{eq:PQconj}.
Then $\mathrm{III}_1$ is equivalent to $P\le Q$, whereas
$\mathrm{III}_2$ is equivalent to $P>Q$.  The equality
$1/q=1/p_1+1/p_2$, equivalently $P=Q$, belongs to
$\mathrm{III}_1$.  In $\mathrm{III}_2$ we define
\begin{equation}\label{eq:R}
 \frac1R=\frac1Q-\frac1P
 =1-q\left(\frac1{p_1}+\frac1{p_2}\right)>0.
\end{equation}

\subsection{Compact bilinear maps and noncompactness radii}
Let $B_{p_i}$ be the closed unit ball of $L^{p_i}(0,\infty)$.  A bounded
bilinear map $T:X_1\times X_2\to L^q$ is compact if
$T(B_{X_1}\times B_{X_2})$ has compact closure.  Let
$\mathcal K_2$ and $\mathcal F_2$ denote, respectively, the compact and the
finite-rank bounded bilinear maps with the same domain and target.  We define
\[
 \|T\|_{\mathrm e}=\inf_{K\in\mathcal K_2}\|T-K\|,
 \qquad
 \distfr(T)=\inf_{F\in\mathcal F_2}\|T-F\|.
\]
The first quantity will be called the \emph{essential distance}. For $q\ge1$
it is the usual quotient distance in the normed bilinear-operator space, while
for $0<q<1$ it is a quasi-norm distance.

For a bounded set $A\subset L^q$, let
\[
 \chi_q(A)=\inf\left\{\varepsilon>0:
 A\subset\bigcup_{j=1}^N\{h:\|h-y_j\|_q<\varepsilon\}
 \text{ for some }y_1,\dots,y_N\right\}.
\]
This is the \emph{radius-based} Hausdorff measure of noncompactness with
quasi-norm balls.  It is not the diameter-based Kuratowski measure, and it
should not be identified with ball-measure conventions that use a different
normalisation.  Consequently, the constant-one identity proved below refers
exactly to this choice of radius.  For $0<q<1$ we also use the complete metric
\[
 d_q(f,g)=\|f-g\|_q^q
\]
and denote by $\chi_{d_q}$ the corresponding metric-radius Hausdorff measure.
The powered essential and finite-rank distances are
\[
 e_q(T)=\inf_{K\in\mathcal K_2}\|T-K\|^q,
 \qquad d_{\mathrm{fr},q}(T)=\inf_{F\in\mathcal F_2}\|T-F\|^q.
\]
Finally, for $n\in\mathbb N=\{1,2,\ldots\}$, define the rank-approximation
sequence
\[
 a_n^{(2)}(T)=\inf\{\|T-F\|:F\in\mathcal F_2,\ \operatorname{rank}F<n\},
\]
where $\operatorname{rank}F$ is the dimension of the linear span of the range.
When the target is Banach, this is exactly the multilinear approximation-number
convention of \cite{FernandezMastyloSilva2013}. The same rank convention and
the same strict index condition $\operatorname{rank}F<n$ are used.  That source
assumes Banach spaces.  Hence, when $0<q<1$, we use the displayed formula as a
manuscript-defined rank-approximation sequence and do not assert the full
external $s$-number framework.  Bilinear measures of noncompactness in
quasi-Banach settings are studied abstractly, for example, in
\cite{BesoyCobos2019}.

\section{Local compactness and exact output-tail localisation}\label{sec:localization}

The localisation mechanism is independent of the Hardy--Steklov weight
characteristics.  We therefore isolate it first for bilinear maps into $L^q$.
For $0<c<d<\infty$ write
\[
 M_{c,d}h=\chi_{(c,d)}h,\qquad
 R_{c,d}h=\chi_{(0,c)\cup(d,\infty)}h.
\]
Thus $M_{c,d}+R_{c,d}=I$ and both operators are contractions on $L^q$.

\begin{lemma}[Uniform tails on compact subsets of $L^q$]\label{lem:compacttails}
Let $K\subset L^q(0,\infty)$ be compact and let
$\theta=\min\{1,q\}$.  Then
\[
 \sup_{h\in K}\|R_{c,d}h\|_q\longrightarrow0
 \qquad(c\downarrow0,\ d\uparrow\infty).
\]
\end{lemma}
\begin{proof}
First take the cofinal sequence $c_n=1/n$, $d_n=n$, for $n\ge2$.  For each fixed $h$,
dominated convergence gives $\|R_{c_n,d_n}h\|_q\to0$.  Given
$\varepsilon>0$, choose a finite $\eta$-net $h_1,\dots,h_N$ for $K$ with
$\eta^\theta<\varepsilon^\theta/2$.  For large $n$, all
$\|R_{c_n,d_n}h_j\|_q^\theta<\varepsilon^\theta/2$.  Since $R_{c_n,d_n}$ is
a contraction and
\[
 \|u+v\|_q^\theta\le \|u\|_q^\theta+\|v\|_q^\theta,
\]
we obtain, for a suitable $j$,
\[
 \|R_{c_n,d_n}h\|_q^\theta
 \le \|h-h_j\|_q^\theta+\|R_{c_n,d_n}h_j\|_q^\theta
 <\varepsilon^\theta.
\]
Monotonicity of the tail sets gives the full directed limit.
\end{proof}

\begin{theorem}[Abstract bilinear output localisation]\label{thm:abstractlocal}
Let $X_1,X_2$ be Banach spaces, $0<q<\infty$, and let
$T:X_1\times X_2\to L^q(0,\infty)$ be bounded and bilinear.  Put
\[
 \gamma_T(c,d)=\|R_{c,d}T\|,
 \qquad
 \tau(T)=\inf_{0<c<d<\infty}\gamma_T(c,d).
\]
Assume that $M_{c,d}T$ is compact for every $0<c<d<\infty$.  Then
\begin{equation}\label{eq:abstract-local-basic}
 T\text{ is compact}\quad\Longleftrightarrow\quad \tau(T)=0,
 \qquad
 \|T\|_{\mathrm e}=\chi_q\bigl(T(B_{X_1}\times B_{X_2})\bigr)=\tau(T).
\end{equation}
If, in addition, every $M_{c,d}T$ is an operator-quasi-norm limit of
finite-rank bilinear maps, then
\begin{equation}\label{eq:abstract-local-fr}
 \distfr(T)=\tau(T),
 \qquad
 a_n^{(2)}(T)\downarrow\tau(T).
\end{equation}
For $0<q<1$ the corresponding powered identities are
\begin{equation}\label{eq:abstract-powered}
 e_q(T)=\chi_{d_q}\bigl(T(B_{X_1}\times B_{X_2})\bigr)=\tau(T)^q,
\end{equation}
and, under the finite-rank approximation hypothesis,
$d_{\mathrm{fr},q}(T)=\tau(T)^q$.
\end{theorem}
\begin{proof}
Let $\theta=\min\{1,q\}$.  If $\tau(T)=0$, monotonicity of
$\gamma_T(c,d)$ allows a cofinal sequence $c_n\downarrow0$, $d_n\uparrow\infty$
with $\gamma_T(c_n,d_n)\to0$.  Since
\[
 \|T-M_{c_n,d_n}T\|=\gamma_T(c_n,d_n),
\]
$T$ is an operator-quasi-norm limit of compact bilinear maps and is compact.
Indeed, finite nets for the image of one sufficiently close compact
approximant give finite nets for $T(B_{X_1}\times B_{X_2})$ by the
$\theta$-triangle inequality.  Conversely, if $T$ is compact, the closure of
$T(B_{X_1}\times B_{X_2})$ is compact, and Lemma~\ref{lem:compacttails} gives
$\gamma_T(c,d)\to0$.  This proves the compactness assertion.

Because $M_{c,d}T$ is compact,
$\|T\|_{\mathrm e}\le\gamma_T(c,d)$ and hence
$\|T\|_{\mathrm e}\le\tau(T)$.  For the reverse estimate let $K$ be any
compact bilinear map.  Then
\[
 \gamma_T(c,d)^\theta
 \le \|R_{c,d}(T-K)\|^\theta+\|R_{c,d}K\|^\theta
 \le \|T-K\|^\theta+\|R_{c,d}K\|^\theta.
\]
The image of the product unit ball under $K$ has compact closure, so the last
term tends to zero by Lemma~\ref{lem:compacttails}.  Taking the directed limit
gives $\tau(T)\le\|T-K\|$, and infimising over $K$ proves
$\|T\|_{\mathrm e}=\tau(T)$.

Assume now local finite-rank approximability.  Given $c,d$ and
$\varepsilon>0$, choose finite-rank $G$ with
$\|M_{c,d}T-G\|<\varepsilon$ and set $F=M_{c,d}G$.  Then $F$ is still
finite-rank, is supported in $(c,d)$, and
$\|M_{c,d}T-F\|<\varepsilon$.  Since every finite-rank map is compact,
$\|T\|_{\mathrm e}\le\distfr(T)$.  Conversely, the two terms in
\[
 T-F=R_{c,d}T+(M_{c,d}T-F)
\]
have disjoint output supports, and therefore, for unit inputs,
\[
 \|(T-F)(x_1,x_2)\|_q^q
 =\|R_{c,d}T(x_1,x_2)\|_q^q
  +\|(M_{c,d}T-F)(x_1,x_2)\|_q^q.
\]
Hence
\[
 \|T-F\|\le(\gamma_T(c,d)^q+\varepsilon^q)^{1/q}.
\]
Letting $\varepsilon\downarrow0$ and taking the directed infimum gives
$\distfr(T)\le\tau(T)$, so equality follows.  Moreover the sets
$\{F:\operatorname{rank}F<n\}$ increase with $n$ and their union is the class
of all finite-rank bilinear maps.  Thus the decreasing sequence
$a_n^{(2)}(T)$ has limit $\distfr(T)$, with no index shift.

It remains to prove the Hausdorff equality without using a Banach triangle
inequality when $q<1$.  Let
\[
 A=T(B_{X_1}\times B_{X_2}).
\]
For fixed $c,d$, the middle image $M_{c,d}A$ is relatively compact.  If
$y_1,\dots,y_N$ form an $\varepsilon$-net for that middle image, then for every
unit pair and a suitable $j$,
\[
 \|T(x_1,x_2)-y_j\|_q^\theta
 \le \|R_{c,d}T(x_1,x_2)\|_q^\theta
    +\|M_{c,d}T(x_1,x_2)-y_j\|_q^\theta.
\]
Consequently
\begin{equation}\label{eq:haus-upper-theta}
 \chi_q(A)\le
 \bigl(\gamma_T(c,d)^\theta+\varepsilon^\theta\bigr)^{1/\theta}.
\end{equation}
First let $\varepsilon\downarrow0$ and then take the directed infimum in
$c,d$ to obtain $\chi_q(A)\le\tau(T)$.

Conversely, suppose that $A$ has an $\varepsilon$-net with centres
$y_1,\dots,y_N$.  Applying $R_{c,d}$ and the $\theta$-triangle inequality gives
\begin{equation}\label{eq:haus-lower-theta}
 \gamma_T(c,d)^\theta
 \le \varepsilon^\theta+
 \max_{1\le j\le N}\|R_{c,d}y_j\|_q^\theta.
\end{equation}
The last maximum tends to zero because the set of centres is finite.  The
directed limit therefore gives $\tau(T)\le\varepsilon$.  Infimising over all
admissible radii proves $\tau(T)\le\chi_q(A)$.

Finally, when $0<q<1$, a $d_q$-ball of metric radius $\rho$ is exactly a
quasi-norm ball of radius $\rho^{1/q}$.  Therefore
$\chi_{d_q}(A)=\chi_q(A)^q$, and the powered distance identities follow by the
same monotonicity of the map $t\mapsto t^q$.  No local convexity is used.
\end{proof}

\subsection*{Verification for the Hardy--Steklov operator}
For $0<c<d<\infty$ define
\[
 E_{c,d}=(0,c)\cup(d,\infty),\qquad
 w_{c,d}=w\chi_{E_{c,d}}.
\]
The output cutoff and the weight cutoff agree exactly,
\begin{equation}\label{eq:cutoff}
 R_{c,d}\B_w=\B_{w_{c,d}}.
\end{equation}
Accordingly,
\begin{equation}\label{eq:gamma-tau}
 \gamma(c,d)=\|R_{c,d}\B_w\|=\|\B_{w_{c,d}}\|,
 \qquad
 \tau(\B_w)=\inf_{0<c<d<\infty}\gamma(c,d).
\end{equation}

\begin{lemma}[Positive inputs and local output integrability]\label{lem:wlocal}
If $\B_w$ is bounded, then $w\in L^q(c,d)$ for every $0<c<d<\infty$.
\end{lemma}
\begin{proof}
Fix $c<d$ and put $I_i=[a_i(c),b_i(d)]$.  Let
$f_i=v_i^{p_i'-1}\chi_{I_i}$.  Then $f_i\in L^{p_i}$ and, for $x\in[c,d]$,
\[
 H_if_i(x)=\int_{a_i(x)}^{b_i(x)}v_i^{p_i'}>0.
\]
The last expression is continuous in $x$ and has a positive minimum on
$[c,d]$.  Since $\B_w(f_1,f_2)\in L^q$, $w\chi_{(c,d)}$ is bounded pointwise by
a constant multiple of $|\B_w(f_1,f_2)|$, which proves the claim.
\end{proof}

\begin{lemma}[Local finite-rank approximation]\label{lem:localcompact}
If $\B_w$ is bounded, then $M_{c,d}\B_w=\chi_{(c,d)}\B_w$ is compact for every
$0<c<d<\infty$.  In fact it is an operator-quasi-norm limit of finite-rank
bilinear maps.
\end{lemma}
\begin{proof}
For $i=1,2$ and $x,y\in[c,d]$, H\"older's inequality gives, uniformly for
$\|f\|_{p_i}\le1$,
\[
 |H_if(x)-H_if(y)|
 \le \left(\int_{I_i(x)\triangle I_i(y)}v_i^{p_i'}\right)^{1/p_i'},
 \qquad I_i(x)=(a_i(x),b_i(x)).
\]
The endpoints are uniformly continuous on $[c,d]$, and the integral of
$v_i^{p_i'}$ is absolutely continuous on the compact interval
$[a_i(c),b_i(d)]$.  Hence the right-hand side tends to zero uniformly as
$|x-y|\to0$.  Also
\[
 |H_if(x)|\le M_i:=
 \left(\int_{a_i(c)}^{b_i(d)}v_i^{p_i'}\right)^{1/p_i'}.
\]

Take a partition $c=x_0<\cdots<x_m=d$ and choose $\xi_j\in[x_{j-1},x_j]$.
Define
\[
 F_{\mathcal P}(f,g)(x)=w(x)\sum_{j=1}^m
 \chi_{(x_{j-1},x_j]}(x)H_1f(\xi_j)H_2g(\xi_j).
\]
By Lemma~\ref{lem:wlocal}, each $w\chi_{(x_{j-1},x_j]}$ belongs to $L^q$, so
$F_{\mathcal P}$ has finite-dimensional range.  If $\omega_i(\delta)$ denotes
the preceding uniform modulus of continuity for $H_i$, then on unit balls
\[
 |H_1f(x)H_2g(x)-H_1f(\xi_j)H_2g(\xi_j)|
 \le M_2\omega_1(|\mathcal P|)+M_1\omega_2(|\mathcal P|).
\]
Consequently
\[
 \|M_{c,d}\B_w-F_{\mathcal P}\|
 \le \|w\|_{L^q(c,d)}
 \bigl(M_2\omega_1(|\mathcal P|)+M_1\omega_2(|\mathcal P|)\bigr)
 \longrightarrow0.
\]
This estimate uses no Minkowski inequality.  Finite-rank bilinear maps are
compact, and an operator-quasi-norm limit of compact bilinear maps is compact by
the finite-net argument used in Theorem~\ref{thm:abstractlocal}.  Thus the two
local hypotheses of that theorem hold.
\end{proof}

\begin{theorem}[Exact tail criterion]\label{thm:tailcompact}
For every bounded $\B_w$,
\[
 \B_w\text{ is compact}\quad\Longleftrightarrow\quad \tau(\B_w)=0.
\]
\end{theorem}
\begin{proof}
Apply Theorem~\ref{thm:abstractlocal} with $T=\B_w$, using
Lemma~\ref{lem:localcompact} and the exact identity \eqref{eq:cutoff}.
\end{proof}

\begin{theorem}[Exact essential-distance and noncompactness identity]\label{thm:exactdistance}
For every bounded $\B_w$,
\begin{equation}\label{eq:exactdistance}
 \|\B_w\|_{\mathrm e}
 =\distfr(\B_w)
 =\chi_q\bigl(\B_w(B_{p_1}\times B_{p_2})\bigr)
 =\tau(\B_w).
\end{equation}
If $0<q<1$, then
\begin{equation}\label{eq:powered}
 e_q(\B_w)=d_{\mathrm{fr},q}(\B_w)
 =\chi_{d_q}\bigl(\B_w(B_{p_1}\times B_{p_2})\bigr)
 =\tau(\B_w)^q.
\end{equation}
\end{theorem}
\begin{proof}
All assertions are the Hardy--Steklov specialisation of
Theorem~\ref{thm:abstractlocal}.  In particular, the Hausdorff part is governed
by \eqref{eq:haus-upper-theta}--\eqref{eq:haus-lower-theta}, so no ordinary
Banach triangle inequality is used when $0<q<1$.
\end{proof}

\begin{corollary}[Rank-approximation limit]\label{cor:approx}
For every bounded $\B_w$,
\[
 a_n^{(2)}(\B_w)\downarrow\tau(\B_w)=\|\B_w\|_{\mathrm e}.
\]
For $q\ge1$ this is the standard multilinear approximation-number sequence of
\cite{FernandezMastyloSilva2013}. For $0<q<1$ it is the rank-approximation
sequence defined in Section~\ref{sec:prelim}.
\end{corollary}
\begin{proof}
This is \eqref{eq:abstract-local-fr} applied to $\B_w$.
\end{proof}

\section{Weighted boundedness characteristics}\label{sec:bounded}

The boundedness results in this section are source-derived.  For Regions
$I,\mathrm{II}_1,\mathrm{II}_2$ they are the corresponding characteristics in
\cite{JainKanjilalStepanovUshakova2018}. The subcritical reduction uses that
paper together with the Oinarov-kernel machinery of
\cite{StepanovUshakova2010MIA,StepanovUshakova2010BJMA}.  Appendix~\ref{app:interface}
identifies the exact source theorems and proves the measure-pullback extension
needed because the derived boundaries are only continuous and strictly
increasing.  All comparisons below have constants depending only on the
exponent and fixed structural parameters, never on an output cutoff.

For a nonnegative measurable output weight $W$, define the following
characteristics.

\subsection{Regions $I$, $\mathrm{II}_1$, and $\mathrm{II}_2$}
In Region $I$, set
\begin{equation}\label{eq:DI}
D_I[W]=\sup_{t,s>0}
\left(\int_{\delta_1(t)\cap\delta_2(s)}W^q\right)^{1/q}
\left(\int_{\Delta_1(t)}v_1^{p_1'}\right)^{1/p_1'}
\left(\int_{\Delta_2(s)}v_2^{p_2'}\right)^{1/p_2'}.
\end{equation}
In Region $\mathrm{II}_1$, where $1/r_2=1/q-1/p_2$, set
\begin{align}
D_{\mathrm{II}_1}[W]
={}&\sup_{t>0}
\left(\int_{\Delta_1(t)}v_1^{p_1'}\right)^{1/p_1'} \notag\\
&\times\left[\int_{\delta_1(t)}
\left(\int_{\delta_1(t)\cap\delta_2(s)}W^q\right)^{r_2/p_2}
\left(\int_{\Delta_2(s)}v_2^{p_2'}\right)^{r_2/p_2'}
W^q(s)\,ds\right]^{1/r_2}.
\label{eq:DII1}
\end{align}
In Region $\mathrm{II}_2$, where $1/r_1=1/q-1/p_1$, we use
\begin{align}
D_{\mathrm{II}_2}^{\mathrm{corr}}[W]
={}&\sup_{s>0}
\left(\int_{\Delta_2(s)}v_2^{p_2'}\right)^{1/p_2'} \notag\\
&\times\left[\int_{\delta_2(s)}
\left(\int_{\delta_1(t)\cap\delta_2(s)}W^q\right)^{r_1/p_1}
\left(\int_{\Delta_1(t)}v_1^{p_1'}\right)^{r_1/p_1'}
W^q(t)\,dt\right]^{1/r_1}.
\label{eq:DII2}
\end{align}
The symmetry-consistent form \eqref{eq:DII2} contains $\Delta_2(s)$ in the
leading $v_2$-factor.  In the displayed Region-$\mathrm{II}_2$ formula of
\cite{JainKanjilalStepanovUshakova2018} that factor is printed with
$\Delta_1(s)$.  Exact interchange of the indices $1\leftrightarrow2$ in
\eqref{eq:DII1} gives \eqref{eq:DII2}. This is the form used throughout the
present paper.

Thus, with
\[
\D_I(W)=D_I[W],\quad
\D_{\mathrm{II}_1}(W)=D_{\mathrm{II}_1}[W],\quad
\D_{\mathrm{II}_2}(W)=D_{\mathrm{II}_2}^{\mathrm{corr}}[W],
\]
we have
\begin{equation}\label{eq:upperbounded}
 \|\B_W\|\asymp \D_{\mathfrak R}(W),
 \qquad \mathfrak R\in\{I,\mathrm{II}_1,\mathrm{II}_2\}.
\end{equation}

\subsection{Subcritical kernels and the central decomposition}
Define
\[
 k_{\ell,W}(y,t)=\int_{b_2^{-1}(t)}^{a_1^{-1}(y)}W(x)^q\,dx,
 \qquad Y_\ell(t)<y<a(t),
\]
\[
 k_{r,W}(t,y)=\int_{b_1^{-1}(y)}^{a_2^{-1}(t)}W(x)^q\,dx,
 \qquad b(t)<y<Y_r(t),
\]
At closed admissible boundary points, the same integral formulas define the
kernels, with value $0$ for a degenerate interval.  Define also
\[
 m_{1,W}(y)=\int_{\Delta_1^{-1}(y)}W^q,
 \qquad
 m_{2,W}(t)=\int_{\Delta_2^{-1}(t)}W^q.
\]
Let
\[
 \Omega=\{t>0:a(t)<b(t)\}=\bigsqcup_{\nu\in\mathcal N}J_\nu.
\]
The component family is at most countable.  On each connected component
$J_\nu$, the continuous function $y_\ell-y_r$ has constant sign. If it
vanished at an interior point of $J_\nu$, then $y_\ell=y_r$, hence $a=b$,
contradicting $J_\nu\subset\Omega$.  We call $J_\nu$ \emph{left type} if
$y_\ell(t)<y_r(t)$ for $t\in J_\nu$, and \emph{right type} if
$y_r(t)<y_\ell(t)$ for $t\in J_\nu$.  Thus it is enough to check the sign
at one point of the component.  On a left central component put
\[
 U_{\nu,W}=\mathcal V_2^{q/r_2},\qquad
 V_{\nu,W}=\mathcal V_1^{q/r_1}m_{1,W};
\]
then
\[
 U_{\nu,W}^Q=\mathcal V_2,\qquad
 V_{\nu,W}^{P'}=\mathcal V_1m_{1,W}^{P'}.
\]
On a right central component put
\[
 U_{\nu,W}=\mathcal V_2^{q/r_2}m_{2,W},\qquad
 V_{\nu,W}=\mathcal V_1^{q/r_1},
\]
so that
\[
 U_{\nu,W}^Q=\mathcal V_2m_{2,W}^Q,\qquad
 V_{\nu,W}^{P'}=\mathcal V_1.
\]

\subsection{Region $\mathrm{III}_1$}\label{subsec:III1-char}
Define
\begin{align}
A_{\ell,0}[W]
={}&\sup_{s>0}\sup_{s\le t\le Y_\ell^{-1}(a(s))}
\left(\int_s^t k_{\ell,W}(Y_\ell(t),x)^{r_2/q}\mathcal V_2(x)\,dx\right)^{q/r_2}
\left(\int_{Y_\ell(t)}^{a(s)}\mathcal V_1(y)\,dy\right)^{q/r_1},
\label{eq:Al0}\\
A_{\ell,1}[W]
={}&\sup_{s>0}\sup_{s\le t\le Y_\ell^{-1}(a(s))}
\left(\int_s^t\mathcal V_2(x)\,dx\right)^{q/r_2}
\left(\int_{Y_\ell(t)}^{a(s)}k_{\ell,W}(y,t)^{r_1/q}\mathcal V_1(y)\,dy\right)^{q/r_1}.
\label{eq:Al1}
\end{align}
Put $D_{\ell,\mathrm{III}_1}[W]=A_{\ell,0}[W]+A_{\ell,1}[W]$.  Similarly,
\begin{align}
A_{r,0}[W]
={}&\sup_{t>0}\sup_{Y_r^{-1}(b(t))\le s\le t}
\left(\int_s^t k_{r,W}(x,Y_r(s))^{r_2/q}\mathcal V_2(x)\,dx\right)^{q/r_2}
\left(\int_{b(t)}^{Y_r(s)}\mathcal V_1(y)\,dy\right)^{q/r_1},
\label{eq:Ar0}\\
A_{r,1}[W]
={}&\sup_{t>0}\sup_{Y_r^{-1}(b(t))\le s\le t}
\left(\int_s^t\mathcal V_2(x)\,dx\right)^{q/r_2}
\left(\int_{b(t)}^{Y_r(s)}k_{r,W}(s,y)^{r_1/q}\mathcal V_1(y)\,dy\right)^{q/r_1}.
\label{eq:Ar1}
\end{align}
Put $D_{r,\mathrm{III}_1}[W]=A_{r,0}[W]+A_{r,1}[W]$.
For each central component let
\begin{equation}\label{eq:A-nu}
\mathfrak A_\nu[W]=
\sup_{\substack{s,t\in J_\nu\\s\le t,\ a(t)\le b(s)}}
\left(\int_s^tU_{\nu,W}(x)^Q\,dx\right)^{1/Q}
\left(\int_{a(t)}^{b(s)}V_{\nu,W}(y)^{P'}\,dy\right)^{1/P'}
\end{equation}
and
\[
 \Adeg[W]=\sup_{\nu\in\mathcal N}\mathfrak A_\nu[W],
\]
with value $0$ when $\Omega=\varnothing$.  Finally,
\begin{equation}\label{eq:DIII1}
\D_{\mathrm{III}_1}(W)=
D_{\ell,\mathrm{III}_1}[W]^{1/q}
+\Adeg[W]^{1/q}
+D_{r,\mathrm{III}_1}[W]^{1/q}.
\end{equation}
The Region-$\mathrm{III}_1$ boundedness criterion is
\begin{equation}\label{eq:III1bounded}
 \|\B_W\|\asymp \D_{\mathrm{III}_1}(W).
\end{equation}
The equality boundary $P=Q$ is included.

\subsection{Region $\mathrm{III}_2$}
When $P>Q$, the same left, central, and right pieces are governed by
$R$-summable block families.  Their complete definitions are given in
Appendix~\ref{app:III2}.  We write
\begin{equation}\label{eq:DIII2}
\D_{\mathrm{III}_2}(W)=
D_{\ell,\mathrm{III}_2}[W]^{1/q}
+\Bdeg[W]^{1/q}
+\bigl(D_{r,\mathrm{III}_2}^{\mathrm{corr}}[W]\bigr)^{1/q}.
\end{equation}
The boundedness criterion is
\begin{equation}\label{eq:III2bounded}
 \|\B_W\|\asymp \D_{\mathrm{III}_2}(W).
\end{equation}
In particular, the right discrete blocks are always used with the kernel
orientations
\begin{equation}\label{eq:right-orientations}
 k_r(x,Y_r(\zeta_k)),\qquad k_r(\zeta_k,y),\qquad
 k_r(x,t),\qquad k_r(t,y).
\end{equation}
A source-interface comment on \eqref{eq:right-orientations} appears in
Appendix~\ref{app:interface}.

\begin{proposition}[Uniform cutoff transfer]\label{prop:tail-transfer}
For every one of the five regions there exists $C_{\mathfrak R}\ge1$, depending
only on the exponent and structural parameters, such that for every
$0<c<d<\infty$,
\[
 C_{\mathfrak R}^{-1}\D_{\mathfrak R}(w_{c,d})
 \le \|\B_{w_{c,d}}\|
 \le C_{\mathfrak R}\D_{\mathfrak R}(w_{c,d}).
\]
The constant is independent of $c,d$, the component and block indices, the
number of components, and the discretisation normalisation.
\end{proposition}

\begin{proof}
Apply the boundedness criteria above with $W=w_{c,d}$.  Multiplication by an
indicator does not alter the endpoint maps, fairways, transformed densities,
or the discretising sequences.  The constants in the source inequalities
therefore remain the same.  For the derived continuous boundaries in the
subcritical pieces, the Stieltjes form described in Appendix~\ref{app:interface}
keeps the constants uniform.
\end{proof}

\section{Upper and mixed regions}\label{sec:uppermixed}

For $\mathfrak R\in\{I,\mathrm{II}_1,\mathrm{II}_2\}$ define the one-sided
characteristic defects
\[
 D_{\mathfrak R}^{0}(c)=\D_{\mathfrak R}(w\chi_{(0,c)}),\qquad
 D_{\mathfrak R}^{\infty}(d)=\D_{\mathfrak R}(w\chi_{(d,\infty)}),
\]
and
\begin{equation}\label{eq:upperdefects}
 \Lambda_{\mathfrak R}^{0}=\lim_{c\downarrow0}D_{\mathfrak R}^{0}(c),
 \qquad
 \Lambda_{\mathfrak R}^{\infty}=\lim_{d\uparrow\infty}D_{\mathfrak R}^{\infty}(d).
\end{equation}
The limits exist in $[0,\infty]$ by monotonicity.  Also let
\[
 \Lambda_{\mathfrak R}=\inf_{0<c<d<\infty}\D_{\mathfrak R}(w_{c,d}).
\]

\begin{lemma}[No-window bridge]\label{lem:nowindowbridge}
For $i=1,2$ write
\[
 \delta_i(t)=(\ell_i(t),u_i(t)),\qquad
 \ell_i(t)=b_i^{-1}(\sigma_i(t)),\quad
 u_i(t)=a_i^{-1}(\sigma_i(t)).
\]
Then $\ell_i,u_i$ are continuous strictly increasing bijections of
$(0,\infty)$ onto itself and $\ell_i(t)<t<u_i(t)$.  If
\[
 \Psi_i(c)=u_i\!\left(\ell_i^{-1}(c)\right),
\]
then
\[
 d\ge\Psi_i(c)
 \quad\Longrightarrow\quad
 \delta_i(t)\text{ cannot meet both }(0,c)\text{ and }(d,\infty)
\]
for any $t>0$.  The family of pairs $(c,d)$ with $d\ge\Psi_i(c)$ is
cofinal as $c\downarrow0$, $d\uparrow\infty$.
\end{lemma}
\begin{proof}
If $\delta_i(t)$ meets $(0,c)$, then $\ell_i(t)<c$, hence
$t<\ell_i^{-1}(c)$.  Monotonicity gives
\[
 u_i(t)<u_i(\ell_i^{-1}(c))=\Psi_i(c)\le d,
\]
so the same window cannot meet $(d,\infty)$.  Cofinality follows by first
choosing $c$ below any prescribed lower cutoff and then choosing
$d\ge\max\{d_0,\Psi_i(c)\}$.
\end{proof}

\begin{lemma}[Two-tail separation]\label{lem:twotail}
For $\mathfrak R\in\{I,\mathrm{II}_1,\mathrm{II}_2\}$,
\begin{equation}\label{eq:twotail}
 \Lambda_{\mathfrak R}
 =\max\{\Lambda_{\mathfrak R}^{0},\Lambda_{\mathfrak R}^{\infty}\}.
\end{equation}
\end{lemma}
\begin{proof}
The lower bound follows from monotonicity because
$w\chi_{(0,c)}\le w_{c,d}$ and $w\chi_{(d,\infty)}\le w_{c,d}$.
For the reverse inequality we work on the cofinal no-window-bridge family.

In Region $I$, every output-weight occurrence in a fixed $(t,s)$ profile is
contained in $\delta_1(t)\cap\delta_2(s)\subset\delta_1(t)$.  Thus
Lemma~\ref{lem:nowindowbridge}, with $i=1$, makes that profile see at most one
of the two tails.  Taking the supremum therefore gives
\[
 D_I[w_{c,d}]=\max\{D_I^0(c),D_I^\infty(d)\}
 \qquad(d\ge\Psi_1(c)).
\]

In Region $\mathrm{II}_1$, both the outer variable and every inner
output-weight integral in \eqref{eq:DII1} lie in the single enclosing window
$\delta_1(t)$.  Hence, on the same cofinal family,
\[
 D_{\mathrm{II}_1}[w_{c,d}]
 =\max\{D_{\mathrm{II}_1}^0(c),D_{\mathrm{II}_1}^\infty(d)\}
 \qquad(d\ge\Psi_1(c)).
\]
In Region $\mathrm{II}_2$, the enclosing window is instead $\delta_2(s)$,
so Lemma~\ref{lem:nowindowbridge} with $i=2$ gives
\[
 D_{\mathrm{II}_2}^{\mathrm{corr}}[w_{c,d}]
 =\max\{D_{\mathrm{II}_2}^0(c),D_{\mathrm{II}_2}^\infty(d)\}
 \qquad(d\ge\Psi_2(c)).
\]
Since each of these no-bridge families is cofinal, taking the directed
infimum proves \eqref{eq:twotail}.  Thus the maximum is an exact
characteristic-level identity, not an operator-norm identity.
\end{proof}

\begin{theorem}[Upper and mixed compactness]\label{thm:uppercompact}
Let $\mathfrak R\in\{I,\mathrm{II}_1,\mathrm{II}_2\}$.  Then
\[
 \B_w\text{ is compact}
 \quad\Longleftrightarrow\quad
 \D_{\mathfrak R}(w)<\infty,
 \quad \Lambda_{\mathfrak R}^{0}=\Lambda_{\mathfrak R}^{\infty}=0.
\]
If $\B_w$ is bounded, then
\begin{equation}\label{eq:upper-essential}
 \|\B_w\|_{\mathrm e}
 \asymp \max\{\Lambda_{\mathfrak R}^{0},\Lambda_{\mathfrak R}^{\infty}\}.
\end{equation}
\end{theorem}
\begin{proof}
Proposition~\ref{prop:tail-transfer}, Lemma~\ref{lem:twotail}, and
Theorems~\ref{thm:tailcompact}--\ref{thm:exactdistance} give the result.  The
identity \eqref{eq:twotail} is exact, while \eqref{eq:upper-essential} is only a
norm equivalence because the regional characteristics are equivalent, rather
than identical, to the operator norm.
\end{proof}

\section{The subcritical supremal region $\mathrm{III}_1$}\label{sec:III1}

The decomposition \eqref{eq:DIII1} contains a left Oinarov piece, a central
component piece, and a right Oinarov piece.  The compactness theorem is most
transparent when the one-sided output cutoffs are retained as the primary
defects.  Define
\begin{align*}
\mathbf D_{\mathrm{III}_1}^{0}(c)
={}&D_{\ell,\mathrm{III}_1}[w\chi_{(0,c)}]^{1/q}
+\Adeg[w\chi_{(0,c)}]^{1/q}
+D_{r,\mathrm{III}_1}[w\chi_{(0,c)}]^{1/q},\\
\mathbf D_{\mathrm{III}_1}^{\infty}(d)
={}&D_{\ell,\mathrm{III}_1}[w\chi_{(d,\infty)}]^{1/q}
+\Adeg[w\chi_{(d,\infty)}]^{1/q}
+D_{r,\mathrm{III}_1}[w\chi_{(d,\infty)}]^{1/q}.
\end{align*}
Set
\begin{equation}\label{eq:III1defects}
 \Lambda_{\mathrm{III}_1}^{0}=\lim_{c\downarrow0}\mathbf D_{\mathrm{III}_1}^{0}(c),
 \qquad
 \Lambda_{\mathrm{III}_1}^{\infty}=\lim_{d\uparrow\infty}\mathbf D_{\mathrm{III}_1}^{\infty}(d),
\end{equation}
and
\[
 \Lambda_{\mathrm{III}_1}(w)=
 \inf_{0<c<d<\infty}\D_{\mathrm{III}_1}(w_{c,d}).
\]
For the bridge proof we use the explicit shorthand
\begin{align*}
 D_\ell^0(c)&=D_{\ell,\mathrm{III}_1}[w\chi_{(0,c)}],&
 D_\ell^\infty(d)&=D_{\ell,\mathrm{III}_1}[w\chi_{(d,\infty)}],\\
 D_r^0(c)&=D_{r,\mathrm{III}_1}[w\chi_{(0,c)}],&
 D_r^\infty(d)&=D_{r,\mathrm{III}_1}[w\chi_{(d,\infty)}].
\end{align*}
Thus no bridge symbol below is implicit.

\begin{theorem}[Profile bridge and two-tail separation]\label{thm:III1bridge}
Let $C_q=2^{\max\{1/q-1,0\}}$.  Then
\begin{equation}\label{eq:III1separation}
\max\{\Lambda_{\mathrm{III}_1}^{0},\Lambda_{\mathrm{III}_1}^{\infty}\}
\le \Lambda_{\mathrm{III}_1}(w)
\le C_q\bigl(\Lambda_{\mathrm{III}_1}^{0}+\Lambda_{\mathrm{III}_1}^{\infty}\bigr)
\le 2C_q\max\{\Lambda_{\mathrm{III}_1}^{0},\Lambda_{\mathrm{III}_1}^{\infty}\}.
\end{equation}
\end{theorem}
\begin{proof}
The first inequality is monotonicity.  The directed upper estimate requires a
profile-level bridge because the five supremal objects in
\eqref{eq:DIII1} need not split for arbitrary cutoffs.

For the two left profiles one may take
\[
 \Psi_{\ell,0}(c)=
 b_2^{-1}\!\left(Y_\ell^{-1}(a(b_2(c)))\right),\qquad
 \Psi_{\ell,1}(c)=a_1^{-1}(a(b_2(c))),
\]
and put $\Psi_\ell=\max\{\Psi_{\ell,0},\Psi_{\ell,1}\}$.  Indeed, a
lower-tail hit in an $A_{\ell,0}$ profile forces its common physical upper
endpoint below $\Psi_{\ell,0}(c)$, while a lower-tail hit in an
$A_{\ell,1}$ profile forces $t<b_2(c)$ and hence every physical upper
endpoint below $\Psi_{\ell,1}(c)$.  Therefore each of
$A_{\ell,0}$ and $A_{\ell,1}$ is exactly one-tail whenever
$d\ge\Psi_\ell(c)$.  After adding the two profiles,
\begin{equation}\label{eq:left-bridge}
 \max\{D_\ell^0(c),D_\ell^\infty(d)\}
 \le D_{\ell,\mathrm{III}_1}[w_{c,d}]
 \le D_\ell^0(c)+D_\ell^\infty(d).
\end{equation}

For a left-type central component define
\[
 \Psi_{ab,\ell}(c)=
 a_1^{-1}\!\left(b(a^{-1}(b_1(c)))\right),
\]
and for a right-type component define
\[
 \Psi_{ab,r}(c)=
 a_2^{-1}\!\left(a^{-1}(b(b_2(c)))\right).
\]
Let $\Psi_{ab}=\max\{\Psi_{ab,\ell},\Psi_{ab,r}\}$.  The profile
condition $a(t)\le b(s)$ propagates one lower-tail mass hit to the displayed
uniform upper bound on every mass interval in the same profile.  Hence, for
every component $J_\nu$, independently of $\nu$,
\[
 \mathfrak A_\nu[w_{c,d}]
 =\max\{\mathfrak A_\nu[w\chi_{(0,c)}],
        \mathfrak A_\nu[w\chi_{(d,\infty)}]\}
 \qquad(d\ge\Psi_{ab}(c)).
\]
Taking the supremum over $\nu$ preserves the maximum,
\begin{equation}\label{eq:central-bridge}
 \Adeg[w_{c,d}]
 =\max\{\Adeg[w\chi_{(0,c)}],\Adeg[w\chi_{(d,\infty)}]\}.
\end{equation}
This is the point at which uniformity over a possibly infinite component
family is essential.

For the two right profiles the corrected right-kernel orientation gives
\[
 \Psi_{r,0}(c)=
 a_2^{-1}\!\left(b^{-1}(Y_r(a_2(c)))\right),\qquad
 \Psi_{r,1}(c)=
 a_2^{-1}\!\left(b^{-1}(b_1(c))\right).
\]
With $\Psi_r=\max\{\Psi_{r,0},\Psi_{r,1}\}$ one obtains
\begin{equation}\label{eq:right-bridge}
 \max\{D_r^0(c),D_r^\infty(d)\}
 \le D_{r,\mathrm{III}_1}[w_{c,d}]
 \le D_r^0(c)+D_r^\infty(d).
\end{equation}

Set
\[
 \Psi(c)=\max\{\Psi_\ell(c),\Psi_{ab}(c),\Psi_r(c)\}.
\]
All maps entering these thresholds are continuous, increasing and cofinal, so
$\Psi(c)\downarrow0$ as $c\downarrow0$ and the family
$d\ge\Psi(c)$ is cofinal.  Combining
\eqref{eq:left-bridge}--\eqref{eq:right-bridge} and using
\[
 (x+y)^{1/q}\le C_q(x^{1/q}+y^{1/q}),\qquad x,y\ge0,
\]
gives, on this cofinal family,
\[
 \D_{\mathrm{III}_1}(w_{c,d})
 \le C_q\left(
 \mathbf D_{\mathrm{III}_1}^{0}(c)+
 \mathbf D_{\mathrm{III}_1}^{\infty}(d)\right).
\]
Choose $c_n\downarrow0$ so that
$\mathbf D_{\mathrm{III}_1}^{0}(c_n)\to\Lambda_{\mathrm{III}_1}^{0}$,
and then choose $d_n\uparrow\infty$ with $d_n\ge\Psi(c_n)$ and
$\mathbf D_{\mathrm{III}_1}^{\infty}(d_n)\to\Lambda_{\mathrm{III}_1}^{\infty}$.
Passing to the limit proves the middle inequality in
\eqref{eq:III1separation}. The last one is elementary.
\end{proof}

\begin{theorem}[Region $\mathrm{III}_1$ compactness and essential distance]\label{thm:III1compact}
In Region $\mathrm{III}_1$,
\[
 \B_w\text{ is compact}
 \quad\Longleftrightarrow\quad
 \D_{\mathrm{III}_1}(w)<\infty,
 \quad \Lambda_{\mathrm{III}_1}^{0}=\Lambda_{\mathrm{III}_1}^{\infty}=0.
\]
For every bounded operator,
\[
 \|\B_w\|_{\mathrm e}
 \asymp \max\{\Lambda_{\mathrm{III}_1}^{0},\Lambda_{\mathrm{III}_1}^{\infty}\}.
\]
The conclusion includes the equality boundary $P=Q$.
\end{theorem}
\begin{proof}
Combine Theorem~\ref{thm:III1bridge} with
Proposition~\ref{prop:tail-transfer} and the exact operator-level
Theorems~\ref{thm:tailcompact} and~\ref{thm:exactdistance}.
\end{proof}

\section{Equivalent endpoint-profile conditions}\label{sec:profiles}

The cutoff defects in \eqref{eq:III1defects} are the primary formulation.  We
now identify them with explicit parameter-endpoint limits.  Let
$\Phi_{\ell,0},\Phi_{\ell,1},\Phi_{r,0},\Phi_{r,1}$ denote the full-weight
expressions inside the double suprema in
\eqref{eq:Al0}--\eqref{eq:Ar1}.

\begin{proposition}[Representative left profile]\label{prop:leftprofile}
For $A_{\ell,0}$,
\begin{align}
\lim_{c\downarrow0}A_{\ell,0}[w\chi_{(0,c)}]
&=\lim_{\varepsilon\downarrow0}
\sup_{\substack{s\le t\le Y_\ell^{-1}(a(s))\\t<\varepsilon}}
\Phi_{\ell,0}(s,t;w),\label{eq:leftlowerprofile}\\
\lim_{d\uparrow\infty}A_{\ell,0}[w\chi_{(d,\infty)}]
&=\lim_{T\uparrow\infty}
\sup_{\substack{s\le t\le Y_\ell^{-1}(a(s))\\s>T}}
\Phi_{\ell,0}(s,t;w).\label{eq:leftupperprofile}
\end{align}
\end{proposition}
\begin{proof}
For $x\in[s,t]$, the kernel $k_{\ell,w}(Y_\ell(t),x)$ uses the physical interval
\[
 (b_2^{-1}(x),a_1^{-1}(Y_\ell(t)))=(b_2^{-1}(x),b_2^{-1}(t)).
\]
If a lower-cutoff profile is nonzero, some such interval meets $(0,c)$, so
$x<b_2(c)$.  Since $s\le x$ and $t\le Y_\ell^{-1}(a(s))$,
\[
 t<Y_\ell^{-1}(a(b_2(c)))=:\varepsilon_{\ell,0}(c)\downarrow0.
\]
This gives one inequality in \eqref{eq:leftlowerprofile}.  Conversely, if
$t<\varepsilon$, every physical interval lies below
$b_2^{-1}(\varepsilon)$. Choosing $c_\varepsilon=b_2^{-1}(\varepsilon)$ makes
the full and lower-cutoff profiles coincide.  This proves
\eqref{eq:leftlowerprofile}.

If an upper-cutoff profile is nonzero, then
$b_2^{-1}(t)>d$, hence $t>b_2(d)$.  Admissibility gives
$Y_\ell(t)\le a(s)$ and therefore
\[
 s>a^{-1}(Y_\ell(b_2(d)))=:T_{\ell,0}(d)\uparrow\infty.
\]
Conversely, if $s>T$, every physical interval begins beyond
$b_2^{-1}(T)$.  This proves \eqref{eq:leftupperprofile}.
\end{proof}

The same reasoning for $k_{\ell,w}(y,t)$ gives
\begin{align}
\lim_{c\downarrow0}A_{\ell,1}[w\chi_{(0,c)}]
&=\lim_{\varepsilon\downarrow0}
\sup_{\substack{s\le t\le Y_\ell^{-1}(a(s))\\t<\varepsilon}}
\Phi_{\ell,1}(s,t;w),\label{eq:left1lower}\\
\lim_{d\uparrow\infty}A_{\ell,1}[w\chi_{(d,\infty)}]
&=\lim_{T\uparrow\infty}
\sup_{\substack{s\le t\le Y_\ell^{-1}(a(s))\\s>T}}
\Phi_{\ell,1}(s,t;w).\label{eq:left1upper}
\end{align}
Indeed a lower-tail hit forces $t<b_2(c)$, while $t<\varepsilon$ places all
intervals below $a_1^{-1}(a(\varepsilon))$. An upper-tail hit forces
$s>a^{-1}(a_1(d))$, while $s>T$ places every interval above
$b_2^{-1}(T)$.

\subsection{Uniform central profiles}
Define
\[
\Phi_\nu(s,t;w)=
\left(\int_s^tU_{\nu,w}(x)^Q\,dx\right)^{1/Q}
\left(\int_{a(t)}^{b(s)}V_{\nu,w}(y)^{P'}\,dy\right)^{1/P'}.
\]

\begin{theorem}[Uniform central endpoint profiles]\label{thm:centralprofile}
One has
\begin{align}
\lim_{c\downarrow0}\Adeg[w\chi_{(0,c)}]
&=\lim_{\varepsilon\downarrow0}
\sup_{\nu}\sup_{\substack{s,t\in J_\nu,\ s\le t,\ a(t)\le b(s)\\t<\varepsilon}}
\Phi_\nu(s,t;w),\label{eq:centrallower}\\
\lim_{d\uparrow\infty}\Adeg[w\chi_{(d,\infty)}]
&=\lim_{T\uparrow\infty}
\sup_{\nu}\sup_{\substack{s,t\in J_\nu,\ s\le t,\ a(t)\le b(s)\\s>T}}
\Phi_\nu(s,t;w).\label{eq:centralupper}
\end{align}
The threshold functions in the proof are independent of $\nu$.
\end{theorem}
\begin{proof}
On a left central component the cutoff enters only through
$m_{1,W}$.  If a lower-tail profile is nonzero, some interval
$\Delta_1^{-1}(y)$ with $y\in[a(t),b(s)]$ meets $(0,c)$.  Hence
$y<b_1(c)$ and therefore
\[
 t<a^{-1}(b_1(c)).
\]
Conversely, if $t<\varepsilon$, then $s\le t$ and every mass parameter obeys
$y\le b(s)\le b(t)<b(\varepsilon)$, so all mass intervals lie in
$(0,a_1^{-1}(b(\varepsilon)))$.  For the upper endpoint, a hit of
$(d,\infty)$ gives $y>a_1(d)$ and hence
$s>b^{-1}(a_1(d))$. Conversely $s>T$ puts every mass interval above
$b_1^{-1}(a(T))$.

On a right central component the cutoff enters only through $m_{2,W}$.  A
lower-tail hit gives some $x\in[s,t]$ with $x<b_2(c)$.  Then
$b(s)\le b(x)<b(b_2(c))$, and $a(t)\le b(s)$ yields
\[
 t<a^{-1}(b(b_2(c))).
\]
Conversely $t<\varepsilon$ puts all mass intervals below
$a_2^{-1}(\varepsilon)$.  For the upper endpoint, a hit gives
$x>a_2(d)$ and therefore $s>b^{-1}(a(a_2(d)))$. Conversely $s>T$ puts all
mass intervals above $b_2^{-1}(T)$.

Thus, for example, one may take the lower threshold
\[
 \varepsilon_{ab}(c)=
 \max\{a^{-1}(b_1(c)),\ a^{-1}(b(b_2(c)))\}
\]
and analogous maxima/minima for the converse and upper thresholds.  Every such
threshold depends only on the global boundary maps and not on the component
index.  The squeeze therefore survives the supremum over $\nu$, proving
\eqref{eq:centrallower}--\eqref{eq:centralupper}.
\end{proof}

\begin{proposition}[Representative right profile]\label{prop:rightprofile}
For $A_{r,0}$,
\begin{align}
\lim_{c\downarrow0}A_{r,0}[w\chi_{(0,c)}]
&=\lim_{\varepsilon\downarrow0}
\sup_{\substack{Y_r^{-1}(b(t))\le s\le t\\t<\varepsilon}}
\Phi_{r,0}(s,t;w),\label{eq:rightlowerprofile}\\
\lim_{d\uparrow\infty}A_{r,0}[w\chi_{(d,\infty)}]
&=\lim_{T\uparrow\infty}
\sup_{\substack{Y_r^{-1}(b(t))\le s\le t\\s>T}}
\Phi_{r,0}(s,t;w).\label{eq:rightupperprofile}
\end{align}
\end{proposition}
\begin{proof}
For $x\in[s,t]$ the relevant physical interval is
\[
 (b_1^{-1}(Y_r(s)),a_2^{-1}(x))=(a_2^{-1}(s),a_2^{-1}(x)).
\]
A lower-tail hit gives $s<a_2(c)$. Since $b(t)\le Y_r(s)$, this implies
$t<b^{-1}(Y_r(a_2(c)))$.  Conversely $t<\varepsilon$ places every interval
below $a_2^{-1}(\varepsilon)$.  An upper-tail hit gives $t>a_2(d)$ and hence
$s>Y_r^{-1}(b(a_2(d)))$. Conversely $s>T$ places the common lower endpoint
above $a_2^{-1}(T)$.  This proves the two identities.
\end{proof}

For $A_{r,1}$ the same argument gives
\begin{align}
\lim_{c\downarrow0}A_{r,1}[w\chi_{(0,c)}]
&=\lim_{\varepsilon\downarrow0}
\sup_{\substack{Y_r^{-1}(b(t))\le s\le t\\t<\varepsilon}}
\Phi_{r,1}(s,t;w),\label{eq:right1lower}\\
\lim_{d\uparrow\infty}A_{r,1}[w\chi_{(d,\infty)}]
&=\lim_{T\uparrow\infty}
\sup_{\substack{Y_r^{-1}(b(t))\le s\le t\\s>T}}
\Phi_{r,1}(s,t;w).\label{eq:right1upper}
\end{align}
Here a lower-tail hit forces $t<b^{-1}(b_1(c))$, whereas a full profile with
$t<\varepsilon$ lies below $a_2^{-1}(\varepsilon)$. At infinity a hit forces
$s>a_2(d)$, while $s>T$ puts the physical intervals above
$b_1^{-1}(b(T))$.

\begin{corollary}[Equivalent endpoint-profile characterisation]\label{cor:profileequiv}
Let $L_{\ell,j}^{0},L_{\ell,j}^{\infty}$,
$L_{r,j}^{0},L_{r,j}^{\infty}$ $(j=0,1)$ be the endpoint limits on the right
of \eqref{eq:leftlowerprofile}--\eqref{eq:right1upper}, and let
$L_{ab}^{0},L_{ab}^{\infty}$ be the two uniform central limits.  Then
\begin{align*}
\Lambda_{\mathrm{III}_1}^{0}
={}&(L_{\ell,0}^{0}+L_{\ell,1}^{0})^{1/q}
 +(L_{ab}^{0})^{1/q}
 +(L_{r,0}^{0}+L_{r,1}^{0})^{1/q},\\
\Lambda_{\mathrm{III}_1}^{\infty}
={}&(L_{\ell,0}^{\infty}+L_{\ell,1}^{\infty})^{1/q}
 +(L_{ab}^{\infty})^{1/q}
 +(L_{r,0}^{\infty}+L_{r,1}^{\infty})^{1/q}.
\end{align*}
Thus the cutoff-defect and parameter-endpoint formulations are equivalent.
\end{corollary}

\section{The strict subcritical region $\mathrm{III}_2$}\label{sec:III2}

In this region $P>Q$ and the positive summability exponent $R$ is given by
\eqref{eq:R}.  The discrete characteristic of Appendix~\ref{app:III2} is an
$\ell^R$ sum over fixed left blocks, fixed central component blocks, and fixed
right blocks.  The following elementary localisation property is the key to
compactness.

\begin{lemma}[Fixed-block extinction]\label{lem:extinction}
For every fixed left, central, or right discrete block in the Region
$\mathrm{III}_2$ characteristic, all physical $W^q$-integration intervals lie
in a compact subinterval of $(0,\infty)$.  Consequently the corresponding
block evaluated at $w_{c,d}$ is exactly zero for all sufficiently small $c$ and
sufficiently large $d$.
\end{lemma}
\begin{proof}
For the left discretisation fix $k$.  The four kernel occurrences are
\[
 k_\ell(t,x),\quad k_\ell(y,t),\quad
 k_\ell(a(\xi_k),x),\quad k_\ell(y,\xi_{k+1}).
\]
Using $Y_\ell(\xi_{k+1})=a(\xi_k)$ and the parameter ranges in
Appendix~\ref{app:III2}, every physical interval is contained in
\[
 \left[b_2^{-1}(\xi_k),\ a_1^{-1}(a(\xi_{k+1}))\right]\Subset(0,\infty).
\]
Hence all four fixed left blocks vanish once the cutoff is outside this compact
interval.

For the right discretisation, the orientations are exactly
\eqref{eq:right-orientations}.  Their physical intervals are, respectively,
\[
 (a_2^{-1}(\zeta_k),a_2^{-1}(x)),\quad
 (b_1^{-1}(y),a_2^{-1}(\zeta_k)),
\]
\[
 (b_1^{-1}(t),a_2^{-1}(x)),\quad
 (b_1^{-1}(y),a_2^{-1}(t)).
\]
The defining ranges and $b(\zeta_{k+1})=Y_r(\zeta_k)$ place all of them in
\[
 \left[b_1^{-1}(b(\zeta_k)),\ a_2^{-1}(\zeta_{k+1})\right]\Subset(0,\infty).
\]
Thus each fixed right block is eventually zero.

Fix now a central pair $(\nu,k)$.  On a left central component, cutoff
dependence occurs only through $m_{1,W}$, and every mass parameter lies in
\[
 [a(\lambda_{\nu,k}),b(\lambda_{\nu,k+1})].
\]
Thus all intervals $\Delta_1^{-1}(s)$ involved in the block lie in
\[
 \left[b_1^{-1}(a(\lambda_{\nu,k})),
 a_1^{-1}(b(\lambda_{\nu,k+1}))\right]\Subset(0,\infty).
\]
On a right central component, cutoff dependence occurs only through $m_{2,W}$,
and the relevant mass parameters lie in
$[\lambda_{\nu,k},\lambda_{\nu,k+1}]$, so all intervals lie in
\[
 \left[b_2^{-1}(\lambda_{\nu,k}),a_2^{-1}(\lambda_{\nu,k+1})\right].
\]
Every central block is therefore eventually exactly zero as well.
\end{proof}

\begin{lemma}[Counting-measure dominated convergence]\label{lem:countingdc}
Let $I$ be countable, $R>0$, and $0\le x_n(i)\le y(i)$ with
$\sum_{i\in I}y(i)^R<\infty$.  If $x_n(i)\to0$ for each fixed $i$, then
$\sum_i x_n(i)^R\to0$.
\end{lemma}
\begin{proof}
This is the dominated convergence theorem on the counting-measure space $I$,
applied to $x_n^R$ and $y^R$.
\end{proof}

\begin{theorem}[Automatic compactness in Region $\mathrm{III}_2$]\label{thm:III2compact}
In Region $\mathrm{III}_2$,
\[
 \B_w\text{ bounded}\quad\Longleftrightarrow\quad\B_w\text{ compact}.
\]
Equivalently, $\D_{\mathrm{III}_2}(w)<\infty$ is the complete compactness
condition.  Every bounded operator satisfies
\[
 \tau(\B_w)=\|\B_w\|_{\mathrm e}=0.
\]
\end{theorem}
\begin{proof}
Only boundedness implies compactness requires proof.  Let
$c_n\downarrow0$, $d_n\uparrow\infty$ and $W_n=w_{c_n,d_n}$.

First, Lemma~\ref{lem:extinction} says that every fixed left block
$B_{\ell,k,j}[W_n]$, central block $\mathfrak b_{\nu,k,j}[W_n]$, and right
block $\widehat B_{r,k,j}[W_n]$ is eventually exactly zero.

Second, $0\le W_n\le w$ and monotonicity of all integrals give
\[
 B_{\ell,k,j}[W_n]^R\le B_{\ell,k,j}[w]^R,
 \quad
 \mathfrak b_{\nu,k,j}[W_n]^R\le\mathfrak b_{\nu,k,j}[w]^R,
\]
\[
 \widehat B_{r,k,j}[W_n]^R\le\widehat B_{r,k,j}[w]^R.
\]
Third, boundedness and \eqref{eq:III2bounded} imply
\[
 \sum_{k,j}B_{\ell,k,j}[w]^R<\infty,
 \qquad
 \sum_{\nu,k,j}\mathfrak b_{\nu,k,j}[w]^R<\infty,
 \qquad
 \sum_{k,j}\widehat B_{r,k,j}[w]^R<\infty.
\]
Hence Lemma~\ref{lem:countingdc} yields
\[
 D_{\ell,\mathrm{III}_2}[W_n]\to0,
 \qquad \Bdeg[W_n]\to0,
 \qquad D_{r,\mathrm{III}_2}^{\mathrm{corr}}[W_n]\to0.
\]
Thus $\D_{\mathrm{III}_2}(W_n)\to0$.  Proposition~\ref{prop:tail-transfer}
gives $\gamma(c_n,d_n)\to0$, hence $\tau(\B_w)=0$.
Theorem~\ref{thm:tailcompact} gives compactness and
Theorem~\ref{thm:exactdistance} gives the exact essential-distance statement.
The mechanism is therefore precisely fixed-block extinction plus full-weight
domination plus $\ell^R$ summability.
\end{proof}

\begin{remark}[The transition at $P=Q$]\label{rem:transition}
At
\[
 \frac1q=\frac1{p_1}+\frac1{p_2}
\]
one has $P=Q$ and $1/Q-1/P=0$ (formally $R=\infty$).  The $\ell^R$ mechanism used in
Theorem~\ref{thm:III2compact} is unavailable.  This equality boundary belongs
to Region $\mathrm{III}_1$, and bounded noncompact examples exist. See
Example~\ref{ex:equality-boundary}.  Under the strict reverse inequality one has
$P>Q$ and the positive exponent $R$ forces tail decay.
\end{remark}

\section{Five-region master classification}\label{sec:master}

\begin{theorem}[Complete compactness classification]\label{thm:master}
Under the assumptions of Section~\ref{sec:prelim}, the five regions have the
following compactness criteria.

\begin{enumerate}[label=\textup{(\roman*)}]
\item If $\mathfrak R\in\{I,\mathrm{II}_1,\mathrm{II}_2\}$, then
\[
 \B_w\text{ compact}
 \iff \D_{\mathfrak R}(w)<\infty
 \text{ and }\Lambda_{\mathfrak R}^{0}=\Lambda_{\mathfrak R}^{\infty}=0.
\]
\item In Region $\mathrm{III}_1$,
\[
 \B_w\text{ compact}
 \iff \D_{\mathrm{III}_1}(w)<\infty
 \text{ and }\Lambda_{\mathrm{III}_1}^{0}=\Lambda_{\mathrm{III}_1}^{\infty}=0.
\]
This includes $P=Q$.
\item In Region $\mathrm{III}_2$,
\[
 \B_w\text{ compact}
 \iff \D_{\mathrm{III}_2}(w)<\infty
 \iff \B_w\text{ bounded}.
\]
\end{enumerate}
\end{theorem}
\begin{proof}
Parts (i), (ii), and (iii) are Theorems~\ref{thm:uppercompact},
~\ref{thm:III1compact}, and~\ref{thm:III2compact}, respectively.
\end{proof}

\begin{table}[ht]
\centering
\small
\renewcommand{\arraystretch}{1.35}
\begin{tabular}{@{}p{1.05cm}p{3.2cm}p{3.0cm}p{3.25cm}p{2.7cm}@{}}
\toprule
Region & Exponent range & Boundedness & Compactness & Essential distance \\
\midrule
$I$ & $1<\max\{p_1,p_2\}\le q$ & $D_I[w]<\infty$ & boundedness + two endpoint limits $0$ & $\asymp\max\{\Lambda_I^0,\Lambda_I^\infty\}$ \\
$\mathrm{II}_1$ & $1<p_1\le q<p_2$ & $D_{\mathrm{II}_1}[w]<\infty$ & boundedness + two endpoint limits $0$ & $\asymp\max\{\Lambda_{\mathrm{II}_1}^0,\Lambda_{\mathrm{II}_1}^\infty\}$ \\
$\mathrm{II}_2$ & $1<p_2\le q<p_1$ & $D_{\mathrm{II}_2}^{\mathrm{corr}}[w]<\infty$ & boundedness + two endpoint limits $0$ & $\asymp\max\{\Lambda_{\mathrm{II}_2}^0,\Lambda_{\mathrm{II}_2}^\infty\}$ \\
$\mathrm{III}_1$ & $0<q<\min\{p_1,p_2\}$, $1/q\le1/p_1+1/p_2$ & $\D_{\mathrm{III}_1}(w)<\infty$ & boundedness + two endpoint limits $0$ & $\asymp\max\{\Lambda_{\mathrm{III}_1}^0,\Lambda_{\mathrm{III}_1}^\infty\}$ \\
$\mathrm{III}_2$ & $0<q<\min\{p_1,p_2\}$, $1/q>1/p_1+1/p_2$ & $\D_{\mathrm{III}_2}(w)<\infty$ & boundedness alone & $0$ exactly \\
\bottomrule
\end{tabular}
\caption{Five-region compactness and essential-distance classification.}
\label{tab:master}
\end{table}

The qualitative dichotomy is therefore
\[
\begin{array}{ll}
 I,\mathrm{II}_1,\mathrm{II}_2,\mathrm{III}_1:
 &\text{boundedness plus lower/upper endpoint vanishing},\\[1mm]
 \mathrm{III}_2:&\text{boundedness alone}.
\end{array}
\]
In the first four regions the controlling tail characteristics do not acquire
the global $R$-summable block structure that forces automatic compactness in
Region $\mathrm{III}_2$.

\section{Essential distances and noncompactness quantities}\label{sec:noncompact}

\begin{theorem}[Five-region essential-distance theorem]\label{thm:essential-master}
For
$\mathfrak R\in\{I,\mathrm{II}_1,\mathrm{II}_2,\mathrm{III}_1\}$, every
bounded operator satisfies
\begin{equation}\label{eq:masteressential}
 \|\B_w\|_{\mathrm e}
 \asymp\max\{\Lambda_{\mathfrak R}^{0},\Lambda_{\mathfrak R}^{\infty}\}.
\end{equation}
For every bounded Region-$\mathrm{III}_2$ operator,
\[
 \|\B_w\|_{\mathrm e}=0.
\]
\end{theorem}
\begin{proof}
For the first three regions use Theorem~\ref{thm:uppercompact}. For Region
$\mathrm{III}_1$ use Theorem~\ref{thm:III1compact}. For Region
$\mathrm{III}_2$ use Theorem~\ref{thm:III2compact}.  At operator level the
identity $\|\B_w\|_{\mathrm e}=\tau(\B_w)$ is exact by
Theorem~\ref{thm:exactdistance}. Only the translation from $\tau$ to regional
weight characteristics introduces equivalence constants in the first four
regions.
\end{proof}

\begin{corollary}[Finite rank, Hausdorff noncompactness, and approximation]\label{cor:allquantities}
For every bounded operator,
\[
 \distfr(\B_w)=\|\B_w\|_{\mathrm e},
 \qquad
 \chi_q\bigl(\B_w(B_{p_1}\times B_{p_2})\bigr)=\|\B_w\|_{\mathrm e},
\]
and
\[
 a_n^{(2)}(\B_w)\downarrow\|\B_w\|_{\mathrm e}.
\]
Consequently the endpoint defects in \eqref{eq:masteressential} control all
three quantities in the first four regions, while all of them vanish in
Region $\mathrm{III}_2$.
\end{corollary}

\begin{corollary}[Powered quasi-Banach formulation]\label{cor:powered}
If $0<q<1$, then
\[
 e_q(\B_w)=d_{\mathrm{fr},q}(\B_w)
 =\chi_{d_q}\bigl(\B_w(B_{p_1}\times B_{p_2})\bigr)
 =\tau(\B_w)^q=\|\B_w\|_{\mathrm e}^q.
\]
In Region $\mathrm{III}_1$ this is comparable to
$\max\{\Lambda_{\mathrm{III}_1}^{0},\Lambda_{\mathrm{III}_1}^{\infty}\}^q$.
In Region $\mathrm{III}_2$ it is exactly zero.
\end{corollary}

\section{Examples and boundary phenomena}\label{sec:examples}

The following examples calibrate the operator-localisation theorem and the two
sides of the subcritical transition.

\begin{example}[Compact support, a direct calibration]\label{ex:compact-support}
Let
\[
 a_i(x)=x,\qquad b_i(x)=\lambda_i x,\qquad v_i=1,
 \qquad w=\chi_{[\alpha,\beta]},
\]
where $\lambda_i>1$ and $0<\alpha<\beta<\infty$.  For
$x\in[\alpha,\beta]$, H\"older's inequality gives
\[
 |H_if(x)|
 \le ((\lambda_i-1)x)^{1/p_i'}\|f\|_{p_i}
 \le ((\lambda_i-1)\beta)^{1/p_i'}\|f\|_{p_i}.
\]
Therefore, for every admissible $0<q<\infty$,
\[
 \|\B_w(f,g)\|_q
 \le (\beta-\alpha)^{1/q}
 \prod_{i=1}^2((\lambda_i-1)\beta)^{1/p_i'}
 \|f\|_{p_1}\|g\|_{p_2}.
\]
Thus boundedness is obtained directly, without invoking any regional
characteristic.  If $c<\alpha$ and $d>\beta$, then $w_{c,d}=0$ and hence
$\gamma(c,d)=0$.  The exact tail criterion therefore makes $\B_w$ compact.
\end{example}

\begin{example}[Equality boundary, scale-invariant noncompactness]\label{ex:equality-boundary}
Take
\[
 p_1=p_2=2,\qquad q=1,\qquad
 a_i(x)=x,\qquad b_i(x)=2x,\qquad v_i=1,\qquad w(x)=x^{-2}.
\]
Then $P=Q=2$, so the point belongs to Region $\mathrm{III}_1$.  The operator is
bounded.  Indeed, with
\[
 A_f(x)=x^{-1}\left|\int_x^{2x}f(t)\,dt\right|,
\]
Cauchy--Schwarz and Fubini give
\[
 \int_0^\infty A_f(x)^2\,dx
 \le \int_0^\infty\frac1x\int_x^{2x}|f(t)|^2\,dt\,dx
 = (\log2)\|f\|_2^2.
\]
Hence
\[
 \|\B_w(f,g)\|_1
 \le \|A_f\|_2\|A_g\|_2
 \le (\log2)\|f\|_2\|g\|_2.
\]

For $\lambda>0$ set
\[
 f_\lambda(t)=\lambda^{1/2}f(\lambda t),\qquad
 g_\lambda(t)=\lambda^{1/2}g(\lambda t).
\]
A change of variables gives
\[
 \|f_\lambda\|_2=\|f\|_2,\qquad
 \|g_\lambda\|_2=\|g\|_2,
\]
and
\[
 \int_x^{2x}f_\lambda(t)\,dt
 =\lambda^{-1/2}\int_{\lambda x}^{2\lambda x}f(u)\,du.
\]
Consequently the output transforms exactly as
\begin{equation}\label{eq:dilation-output}
 \B_w(f_\lambda,g_\lambda)(x)
 =\lambda\,\B_w(f,g)(\lambda x),
\end{equation}
so its $L^1$ norm is unchanged.

Fix $0<c<d<\infty$ and $\eta>0$.  Choose unit inputs $f,g$ such that, with
$y=\B_w(f,g)$,
\[
 \|y\|_1>\|\B_w\|-\eta.
\]
Choose $0<A<B<\infty$ with
$\int_A^B|y|>\|y\|_1-\eta$.  If $\lambda>B/c$, then by
\eqref{eq:dilation-output}
\[
 \int_0^c|\B_w(f_\lambda,g_\lambda)(x)|\,dx
 =\int_0^{\lambda c}|y(u)|\,du
 >\|\B_w\|-2\eta.
\]
If instead $0<\lambda<A/d$, then
\[
 \int_d^\infty|\B_w(f_\lambda,g_\lambda)(x)|\,dx
 =\int_{\lambda d}^\infty|y(u)|\,du
 >\|\B_w\|-2\eta.
\]
The input norms are unchanged in both cases.  Since the tail operator is a
contraction, letting $\eta\downarrow0$ proves the exact identity
\[
 \gamma(c,d)=\|\B_w\|\qquad(0<c<d<\infty).
\]
Thus $\tau(\B_w)=\|\B_w\|>0$ and the operator is noncompact.  This shows that
the automatic-compactness conclusion valid for $P>Q$ cannot be extended to the
boundary $P=Q$.
\end{example}

\begin{example}[Bounded Region-$\mathrm{III}_2$ operator with noncompact support]\label{ex:III2-fullsupport}
Let
\[
 p_1=p_2=2,\qquad q=\frac12,\qquad
 a_i(x)=x,\qquad b_i(x)=2x,\qquad v_i=1,
\]
and
\[
 w(x)=e^{-x-1/x}.
\]
Then $P=4$, $Q=4/3$, and $R=2$.  Since
\[
 |H_if(x)|\le x^{1/2}\|f\|_2,
\]
we have
\[
 |\B_w(f,g)(x)|^{1/2}
 \le (xw(x))^{1/2}\|f\|_2^{1/2}\|g\|_2^{1/2}.
\]
As
\[
 \int_0^\infty (xw(x))^{1/2}\,dx<\infty,
\]
$\B_w:L^2\times L^2\to L^{1/2}$ is bounded.  Its output weight is positive on
the whole half-line, so compactness is not caused by compact support.  The
Region-$\mathrm{III}_2$ theorem nevertheless makes the operator compact.
\end{example}

\begin{example}[A genuine central obstruction]\label{ex:central}
Let
\[
 a_1(x)=a_2(x)=x,\qquad b_1(x)=2x,\qquad b_2(x)=4x,
\]
\[
 p_1=p_2=2,\qquad q=1,\qquad v_1=v_2=1,\qquad w(x)=x^{-2}.
\]
Then
\[
 y_\ell(t)=b_1(b_2^{-1}(t))=\frac t2,\qquad
 y_r(t)=a_1(a_2^{-1}(t))=t,
\]
and hence
\[
 a(t)=\frac t2,\qquad b(t)=t.
\]
The unique central component is of left type.  Moreover
\[
 m_{1,w}(y)
 =\int_{b_1^{-1}(y)}^{a_1^{-1}(y)}x^{-2}\,dx
 =\int_{y/2}^{y}x^{-2}\,dx
 =\frac1y.
\]
Substitution into the central profile \eqref{eq:A-nu} gives
\begin{equation}\label{eq:central-calibration-profile}
 \sup_{s\le t\le2s}
 (t-s)^{1/2}
 \left(\int_{t/2}^{s}y^{-2}\,dy\right)^{1/2}.
\end{equation}
Choosing $t=3s/2$ yields
\[
 (t-s)^{1/2}
 \left(\int_{t/2}^{s}y^{-2}\,dy\right)^{1/2}
 =\left(\frac s2\right)^{1/2}
  \left(\frac1{3s}\right)^{1/2}
 =\frac1{\sqrt6}.
\]
Thus $\Adeg[w]\ge1/\sqrt6$.

The same constant survives at both output endpoints.  For a lower cutoff
$W=w\chi_{(0,c)}$, choose $s<c$ and $t=3s/2$.  Every interval
$[y/2,y]$ with $y\in[t/2,s]$ then lies below $c$, so
$m_{1,W}(y)=1/y$ throughout the profile and
\[
 \Adeg[w\chi_{(0,c)}]\ge\frac1{\sqrt6}.
\]
For an upper cutoff $W=w\chi_{(d,\infty)}$, choose
$s>8d/3$ and again $t=3s/2$.  Then the smallest lower endpoint of the physical
mass intervals is $(t/2)/2=3s/8>d$, so once more
$m_{1,W}(y)=1/y$ throughout and
\[
 \Adeg[w\chi_{(d,\infty)}]\ge\frac1{\sqrt6}.
\]
Hence both central endpoint defects are nonzero.  Thus the central term
$\mathcal A_{\rm deg}$ has a persistent nonzero endpoint defect in Region
$\mathrm{III}_1$.
\end{example}

\appendix
\section{The Region $\mathrm{III}_2$ discrete characteristic}\label{app:III2}

For completeness we record the block characteristic used in
Section~\ref{sec:III2}.  Let $R$ be given by \eqref{eq:R}.

\subsection{Left blocks}
Let
\[
 \xi_0=1,\qquad
 \xi_{k+1}=Y_\ell^{-1}\circ a(\xi_k),\qquad
 \xi_{k-1}=a^{-1}\circ Y_\ell(\xi_k),
\]
so that $Y_\ell(\xi_{k+1})=a(\xi_k)$.  Define
\begin{align*}
B_{\ell,k,1}[W]^R
={}&\int_{Y_\ell(\xi_k)}^{a(\xi_k)}
\left[\int_{\xi_k}^{Y_\ell^{-1}(t)}
 k_{\ell,W}(t,x)^{r_2/q}\mathcal V_2(x)\,dx\right]^{Rq/r_2}\\
&\qquad\times
\left[\int_t^{a(\xi_k)}\mathcal V_1(y)\,dy\right]^{R(r_2-q)/r_2}
\mathcal V_1(t)\,dt,
\end{align*}
\begin{align*}
B_{\ell,k,2}[W]^R
={}&\int_{\xi_k}^{\xi_{k+1}}
\left[\int_{\xi_k}^{t}\mathcal V_2(x)\,dx\right]^{R(r_1-q)/r_1}\\
&\qquad\times
\left[\int_{Y_\ell(t)}^{a(\xi_k)}
 k_{\ell,W}(y,t)^{r_1/q}\mathcal V_1(y)\,dy\right]^{Rq/r_1}
\mathcal V_2(t)\,dt,
\end{align*}
\begin{align*}
B_{\ell,k,3}[W]^R
={}&\int_{Y_\ell(\xi_{k+1})}^{a(\xi_{k+1})}
\left[\int_{a^{-1}(t)}^{\xi_{k+1}}
 k_{\ell,W}(a(\xi_k),x)^{r_2/q}\mathcal V_2(x)\,dx\right]^{Rq/r_2}\\
&\qquad\times
\left[\int_{Y_\ell(\xi_{k+1})}^{t}\mathcal V_1(y)\,dy\right]^{R(r_2-q)/r_2}
\mathcal V_1(t)\,dt,
\end{align*}
\begin{align*}
B_{\ell,k,4}[W]^R
={}&\int_{\xi_k}^{\xi_{k+1}}
\left[\int_t^{\xi_{k+1}}\mathcal V_2(x)\,dx\right]^{R(r_1-q)/r_1}\\
&\qquad\times
\left[\int_{a(\xi_k)}^{a(t)}
 k_{\ell,W}(y,\xi_{k+1})^{r_1/q}\mathcal V_1(y)\,dy\right]^{Rq/r_1}
\mathcal V_2(t)\,dt.
\end{align*}
Then
\[
 D_{\ell,\mathrm{III}_2}[W]=
 \left(\sum_{k\in\mathbb Z}\sum_{j=1}^4B_{\ell,k,j}[W]^R\right)^{1/R}.
\]

\subsection{Central blocks}
For each central component $J_\nu$ define
\[
 T_\nu=a^{-1}\circ b:J_\nu\to J_\nu.
\]
Choose $\lambda_{\nu,0}\in J_\nu$ and set
\begin{equation}\label{eq:central-biinfinite}
 \lambda_{\nu,k}=T_\nu^k(\lambda_{\nu,0}),\qquad k\in\mathbb Z,
\end{equation}
where negative powers use
$T_\nu^{-1}=b^{-1}\circ a$.  Since $a(t)<b(t)$ on $J_\nu$ and both maps are
strictly increasing,
\[
 T_\nu(t)>t,\qquad
 a(\lambda_{\nu,k+1})=b(\lambda_{\nu,k}).
\]
The map $T_\nu$ is an increasing homeomorphism of $J_\nu$ onto itself.  Indeed,
if a point between $t$ and $T_\nu(t)$ satisfied $a=b$, strict monotonicity would
give the contradiction
$b(t)<b(u)=a(u)<a(T_\nu(t))=b(t)$. The same argument applies to the inverse.
Thus the iterates remain in the same component.  The sequence is strictly
increasing in $k$, and
\[
 \lambda_{\nu,k}\longrightarrow\inf J_\nu\quad(k\to-\infty),
 \qquad
 \lambda_{\nu,k}\longrightarrow\sup J_\nu\quad(k\to+\infty).
\]
A finite interior limit would be a fixed point of $T_\nu$, impossible because
$T_\nu(t)>t$.  Consequently the blocks
$[\lambda_{\nu,k},\lambda_{\nu,k+1}]$, $k\in\mathbb Z$, cover the entire
component (up to shared endpoints).  This is the componentwise version of the
bi-infinite discretisation in Section~3.2, equations (3.5)--(3.6), of
\cite{StepanovUshakova2010MIA}.

Define
\begin{align*}
\mathfrak b_{\nu,k,1}[W]^R
={}&\int_{a(\lambda_{\nu,k})}^{a(\lambda_{\nu,k+1})}
\left(\int_{\lambda_{\nu,k}}^{a^{-1}(t)}U_{\nu,W}(x)^Q\,dx\right)^{R/Q}\\
&\qquad\times
\left(\int_t^{a(\lambda_{\nu,k+1})}V_{\nu,W}(y)^{P'}\,dy\right)^{R/Q'}
V_{\nu,W}(t)^{P'}\,dt,
\end{align*}
\begin{align*}
\mathfrak b_{\nu,k,2}[W]^R
={}&\int_{\lambda_{\nu,k}}^{\lambda_{\nu,k+1}}
\left(\int_{\lambda_{\nu,k}}^tU_{\nu,W}(x)^Q\,dx\right)^{R/P}\\
&\qquad\times
\left(\int_{a(t)}^{a(\lambda_{\nu,k+1})}V_{\nu,W}(y)^{P'}\,dy\right)^{R/P'}
U_{\nu,W}(t)^Q\,dt,
\end{align*}
\begin{align*}
\mathfrak b_{\nu,k,3}[W]^R
={}&\int_{b(\lambda_{\nu,k})}^{b(\lambda_{\nu,k+1})}
\left(\int_{b^{-1}(t)}^{\lambda_{\nu,k+1}}U_{\nu,W}(x)^Q\,dx\right)^{R/Q}\\
&\qquad\times
\left(\int_{b(\lambda_{\nu,k})}^{t}V_{\nu,W}(y)^{P'}\,dy\right)^{R/Q'}
V_{\nu,W}(t)^{P'}\,dt,
\end{align*}
\begin{align*}
\mathfrak b_{\nu,k,4}[W]^R
={}&\int_{\lambda_{\nu,k}}^{\lambda_{\nu,k+1}}
\left(\int_t^{\lambda_{\nu,k+1}}U_{\nu,W}(x)^Q\,dx\right)^{R/P}\\
&\qquad\times
\left(\int_{b(\lambda_{\nu,k})}^{b(t)}V_{\nu,W}(y)^{P'}\,dy\right)^{R/P'}
U_{\nu,W}(t)^Q\,dt.
\end{align*}
Set
\[
 \Bdeg[W]=
 \left(\sum_{\nu\in\mathcal N}\sum_{k\in\mathbb Z}\sum_{j=1}^4
 \mathfrak b_{\nu,k,j}[W]^R\right)^{1/R},
\]
with value $0$ if $\Omega=\varnothing$.

\subsection{Right blocks}
Let
\[
 \zeta_0=1,\qquad
 \zeta_{k+1}=b^{-1}\circ Y_r(\zeta_k),\qquad
 \zeta_{k-1}=Y_r^{-1}\circ b(\zeta_k),
\]
so $b(\zeta_{k+1})=Y_r(\zeta_k)$.  Define
\begin{align*}
\widehat B_{r,k,1}[W]^R
={}&\int_{b(\zeta_k)}^{Y_r(\zeta_k)}
\left[\int_{\zeta_k}^{b^{-1}(t)}
 k_{r,W}(x,Y_r(\zeta_k))^{r_2/q}\mathcal V_2(x)\,dx\right]^{Rq/r_2}\\
&\qquad\times
\left[\int_t^{Y_r(\zeta_k)}\mathcal V_1(y)\,dy\right]^{R(r_2-q)/r_2}
\mathcal V_1(t)\,dt,
\end{align*}
\begin{align*}
\widehat B_{r,k,2}[W]^R
={}&\int_{\zeta_k}^{\zeta_{k+1}}
\left[\int_{\zeta_k}^{t}\mathcal V_2(x)\,dx\right]^{R(r_1-q)/r_1}\\
&\qquad\times
\left[\int_{b(t)}^{b(\zeta_{k+1})}
 k_{r,W}(\zeta_k,y)^{r_1/q}\mathcal V_1(y)\,dy\right]^{Rq/r_1}
\mathcal V_2(t)\,dt,
\end{align*}
\begin{align*}
\widehat B_{r,k,3}[W]^R
={}&\int_{b(\zeta_{k+1})}^{Y_r(\zeta_{k+1})}
\left[\int_{Y_r^{-1}(t)}^{\zeta_{k+1}}
 k_{r,W}(x,t)^{r_2/q}\mathcal V_2(x)\,dx\right]^{Rq/r_2}\\
&\qquad\times
\left[\int_{b(\zeta_{k+1})}^{t}\mathcal V_1(y)\,dy\right]^{R(r_2-q)/r_2}
\mathcal V_1(t)\,dt,
\end{align*}
\begin{align*}
\widehat B_{r,k,4}[W]^R
={}&\int_{\zeta_k}^{\zeta_{k+1}}
\left[\int_t^{\zeta_{k+1}}\mathcal V_2(x)\,dx\right]^{R(r_1-q)/r_1}\\
&\qquad\times
\left[\int_{Y_r(\zeta_k)}^{Y_r(t)}
 k_{r,W}(t,y)^{r_1/q}\mathcal V_1(y)\,dy\right]^{Rq/r_1}
\mathcal V_2(t)\,dt.
\end{align*}
Thus
\[
 D_{r,\mathrm{III}_2}^{\mathrm{corr}}[W]=
 \left(\sum_{k\in\mathbb Z}\sum_{j=1}^4
 \widehat B_{r,k,j}[W]^R\right)^{1/R}.
\]
Combining the three terms gives \eqref{eq:DIII2}.

\section{Continuous-boundary and source interfaces}\label{app:interface}

The subcritical reconstruction uses one-sided Hardy--Oinarov criteria after
coordinate changes.  The printed source reductions are commonly written for
differentiable boundary functions, whereas the derived maps
$a,b,Y_\ell,Y_r$ in this paper are only continuous and strictly increasing.
This appendix records precisely why the source criteria remain available
without introducing an inverse derivative.

The measure-theoretic source is Prokhorov \cite{Prokhorov2007}.  In the
notation of that paper, Theorem~1 treats $1<p\le q<\infty$ and Theorem~2
treats $1<q<p<\infty$ for the Hardy--Oinarov operator with positive Borel
$\sigma$-finite measures $\lambda,\mu$, with $\nu=\lambda$ and source
multiplier $w\equiv1$. The coefficient functions $u,v$ are nonnegative Borel
functions, while the kernel is nonnegative, $\mu\times\lambda$-measurable,
and satisfies the ordered Oinarov relation.  In each application below the
input integration measure is the relevant Lebesgue or pullback measure
$\lambda$, the input norm uses $\nu=\lambda$, and the remaining input and
output densities are absorbed into Prokhorov's coefficient functions $u$ and
$v$.  Thus no extra measure or multiplier hypothesis is being suppressed.  In
the notation of the present subcritical reduction we apply those theorems with
\[
 p=P,\qquad q=Q.
\]
Thus Theorem~1 covers Region $\mathrm{III}_1$, where $1<P\le Q<\infty$, and
Theorem~2 covers Region $\mathrm{III}_2$, where $1<Q<P<\infty$.  The printed
one-sided formulas used in the variable-boundary reduction are Theorem~2.2 and
Corollaries~2.1--2.2 of \cite{StepanovUshakova2010MIA}.  Its
Lemma~3.1 supplies the block-diagonal norm aggregation, equations
(3.5)--(3.6) in Section~3.2 supply the bi-infinite discretisation, and
Theorems~3.1--3.2 assemble the two one-sided pieces.  Related one-sided
lower/upper estimates are Theorems~2.5--2.6 of
\cite{StepanovUshakova2010BJMA}.

\begin{lemma}[Continuous-boundary one-sided interface]\label{lem:continuous-boundary}
The one-sided Oinarov-kernel criteria used in
\eqref{eq:Al0}--\eqref{eq:Ar1}, and their $R$-summable analogues in
Appendix~\ref{app:III2}, remain valid when a derived boundary is continuous and
strictly increasing rather than differentiable.  The comparison constants
depend only on $P,Q$ and the fixed Oinarov constant. In particular they are
independent of every output cutoff and every component index.
\end{lemma}
\begin{proof}
We match the source hypotheses explicitly.  Let $I$ be an interval and let
$\beta:I\to\beta(I)$ be continuous and strictly increasing.  Since $\beta$ is
a homeomorphism onto an interval, define on Borel sets
\begin{equation}\label{eq:pullback-measure}
 \mu_\beta(E)=|\beta(E)|.
\end{equation}
Equivalently, $\mu_\beta$ is the pullback of Lebesgue measure under this
homeomorphism.  It is a positive Borel measure.  It is locally finite because
for every compact interval $[u,v]\Subset I$,
\[
 \mu_\beta([u,v])=\beta(v)-\beta(u)<\infty,
\]
and hence it is $\sigma$-finite by a compact exhaustion of $I$.  This verifies
the measure hypotheses in Theorems~1 and~2 of \cite{Prokhorov2007}.

If $F$ is nonnegative and measurable, then
\begin{equation}\label{eq:stieltjes-pullback}
 \int_I F(t)\,d\mu_\beta(t)
 =\int_{\beta(I)}F(\beta^{-1}(y))\,dy.
\end{equation}
All weights occurring in the reduction are measurable. After changing values
on null sets they may be represented by Borel functions, as required in the
source theorem.  This does not change any characteristic.  If the original
kernel $k$ is Borel (in the source applications it is continuous or is a
measurable integral kernel built from the weights), then
\[
 \widetilde k(x,t)=k(x,\beta(t))
\]
is product-measurable.  Thus the source measurability assumptions are also
preserved.

Consider first a moving upper boundary.  Formula
\eqref{eq:stieltjes-pullback} gives exactly
\begin{equation}\label{eq:upper-pullback}
 \int_{\beta(c)}^{\beta(x)}k(x,y)f(y)V(y)\,dy
 =\int_{(c,x)}k(x,\beta(t))f(\beta(t))V(\beta(t))\,d\mu_\beta(t).
\end{equation}
No identity of the form $dy=\beta'(t)\,dt$ is used.  Suppose the one-sided
kernel satisfies the boundary-adapted Oinarov relation
\[
 D^{-1}k(x,y)
 \le k(x,\beta(z))+k(z,y)
 \le Dk(x,y),
 \qquad c\le z\le x,\quad y\le\beta(z).
\]
Putting $y=\beta(t)$ shows that $\widetilde k$ satisfies Prokhorov's ordered
Oinarov relation for $t\le z\le x$ with the \emph{same} constant $D$.
Because $\beta$ is order preserving, all splitting points and order relations
in the source proof are unchanged.  Hence Theorem~1 or~2 of
\cite{Prokhorov2007}, according as $P\le Q$ or $Q<P$, applies to
\eqref{eq:upper-pullback}.  Undoing the pullback gives exactly the
moving-upper-boundary orientation of Corollary~2.1 in
\cite{StepanovUshakova2010MIA}, now without differentiability of $\beta$.

For a moving lower boundary, the same pullback with the lower boundary
$\alpha$ produces a Copson-ordered integral over $(x,d)$.  Regard that interval
with the opposite order, or equivalently push both source measures forward by
an order-reversing homeomorphism onto an ordinary interval.  Positive Borel
$\sigma$-finite measures remain such, and the Oinarov splitting relation is
carried into the standard ordered form with the same constant $D$.  Applying
Prokhorov's Theorem~1 or~2 and undoing the order reversal yields the
moving-lower-boundary orientation of Corollary~2.2 in
\cite{StepanovUshakova2010MIA}.  Again no derivative, inverse derivative, or
absolute-continuity assumption on the boundary is introduced.

It remains to check uniformity.  The coordinate changes above are exact
measure-space isometries.  The constants in Prokhorov's Theorems~1--2 depend
only on the two exponents and the Oinarov constant.  Multiplying the output
weight by $\chi_{(0,c)}$, $\chi_{(d,\infty)}$, or
$\chi_{(0,c)\cup(d,\infty)}$ changes only the measurable weights in the
characteristic and does not change the structural Oinarov constant of the
derived kernels.  Therefore the source constants are independent of the
cutoff.  The same argument on a component interval $J_\nu$ uses only an order
homeomorphism and the corresponding pullback measures, so the constants are
also independent of $\nu$.  Lemma~3.1 and Theorems~3.1--3.2 of
\cite{StepanovUshakova2010MIA} can consequently be assembled componentwise with
one uniform constant.
\end{proof}

Apply the same generic pullback to each full-window/mixed fairway. For $i=1,2$,
\[
 \mu_{\sigma_i}(E)=|\sigma_i(E)|,\qquad
 \widetilde a_i=\sigma_i^{-1}\circ a_i,\qquad
 \widetilde b_i=\sigma_i^{-1}\circ b_i.
\]
Relative to $\mu_{\sigma_i}$ the identity is the splitting point, since
\[
\begin{aligned}
 \int_{\widetilde a_i(x)}^x
 v_i(\sigma_i(t))^{p_i'}\,d\mu_{\sigma_i}(t)
 &=\int_{a_i(x)}^{\sigma_i(x)}v_i(y)^{p_i'}\,dy\\
 &=\int_{\sigma_i(x)}^{b_i(x)}v_i(y)^{p_i'}\,dy\\
 &=\int_x^{\widetilde b_i(x)}
 v_i(\sigma_i(t))^{p_i'}\,d\mu_{\sigma_i}(t).
\end{aligned}
\]
Thus the equal-mass relation is transported exactly, without $\sigma_i'$ and
without changing fairway regularity, any characteristic, or any comparison
constant.

For clarity, the central piece requires no additional source assumption.  On
each $J_\nu$ its kernel is $k\equiv1$, which is Borel and satisfies the standard
Oinarov relation with a fixed constant, for example $D=2$.  In Region
$\mathrm{III}_1$ the effective exponents obey $1<P\le Q<\infty$. Prokhorov's
Theorem~1 supplies the one-sided estimates, and the two-boundary assembly in
Theorem~3.1 of \cite{StepanovUshakova2010MIA}, specialised to $k\equiv1$,
yields the supremal central criterion \eqref{eq:A-nu}.  In Region
$\mathrm{III}_2$ one has $1<Q<P<\infty$. Prokhorov's Theorem~2 supplies the
integrable one-sided criteria, and the block assembly of Section~3 of
\cite{StepanovUshakova2010MIA} yields the four central blocks in
Appendix~\ref{app:III2}.  The bi-infinite component discretisation was made
explicit in \eqref{eq:central-biinfinite}.  Thus the central $k\equiv1$ case is
included with constants independent of the component.  The lower-regularity
Hardy--Steklov results of \cite{BernardisMartinReyesOrtega2007} are consistent
with this continuous-boundary formulation, but they are not used here as a
substitute for the Prokhorov hypothesis check.

\begin{remark}[Source-interface checklist]\label{rem:source-checklist}
The preceding argument verifies, in order, positivity and Borel
$\sigma$-finiteness of the pullback measures, local finiteness, measurable
pulled-back weights and kernels, preservation of the Oinarov relation and its
constant, preservation of order and splitting points, both moving-boundary
orientations, the effective exponent ranges $1<P\le Q$ and $1<Q<P$, absence of
inverse derivatives, cutoff-independent constants, component-independent
constants, stability under component coordinate changes, and the central
$k\equiv1$ case.  These are exactly the hypotheses needed at the source
interface.
\end{remark}

\begin{remark}[Right-kernel orientation]
The right kernel is
\[
 k_r(t,y)=\int_{b_1^{-1}(y)}^{a_2^{-1}(t)}W(x)^q\,dx.
\]
The variable roles in the right discrete decomposition therefore produce the
four occurrences listed in \eqref{eq:right-orientations}.  These are the forms
used in Appendix~\ref{app:III2}.  They are also the forms obtained by tracing
the right Oinarov identity through the block variables.  This explicit
variable-role derivation is used here rather than importing the displayed
right-block formula of \cite{JainKanjilalStepanovUshakova2018} without
reconciliation.
The printed right-block display in that source is malformed. It uses the wrong
right-kernel family/orientation and reversed argument ordering, while blocks 2
and 4 use $Q$ where the one-sided block theorem requires $P'=r_1/q$.
Appendix~\ref{app:III2} is derived from that theorem. This source-interface
reconciliation is not presented as a formal erratum to the earlier paper.
\end{remark}

\section*{Use of generative AI}

OpenAI's ChatGPT, including ChatGPT Work, and Codex were used as assistive
tools during manuscript preparation for mathematical consistency checking,
language refinement, readability improvement, and LaTeX preparation. The
author independently reviewed the mathematical arguments, references, and
final manuscript and takes full responsibility for all statements, proofs,
and conclusions.

% The bibliography is maintained inline. The detached .bib is archival metadata.

\end{document}